\documentclass[12pt]{article}
\usepackage{amssymb}
\usepackage{amsthm}
\usepackage{amsmath}
\usepackage{bbm}
\usepackage{amscd}
\usepackage{stmaryrd}
\usepackage{array}
\usepackage{url}
\usepackage{hyperref}
\usepackage[latin2]{inputenc}
\usepackage{t1enc}
\usepackage{enumerate}
\usepackage[mathscr]{eucal}
\usepackage[british]{babel}
\usepackage[dvips]{graphicx}
\usepackage[dvips]{epsfig}
\usepackage{epsf}
\usepackage{color}
\usepackage{indentfirst}
\usepackage[all]{xy}

\newcommand{\Ker}{\operatorname{Ker}}
\newcommand{\GQpD}{G_{\mathbb{Q}_p,\Delta}}
\newcommand{\GQpa}{G_{\mathbb{Q}_p,\alpha}}
\newcommand{\HQpD}{H_{\mathbb{Q}_p,\Delta}}
\newcommand{\HQpa}{H_{\mathbb{Q}_p,\alpha}}
\newcommand{\GQpDa}{G_{\mathbb{Q}_p,\Delta\setminus\{\alpha\}}}
\newcommand{\HQpDa}{H_{\mathbb{Q}_p,\Delta\setminus\{\alpha\}}}
\newcommand{\RepDFp}{\mathrm{Rep}_{\mathbb{F}_p}(G_{\mathbb{Q}_p,\Delta})}
\newcommand{\RepDZp}{\mathrm{Rep}_{\mathbb{Z}_p}(G_{\mathbb{Q}_p,\Delta})}
\newcommand{\RepDQp}{\mathrm{Rep}_{\mathbb{Q}_p}(G_{\mathbb{Q}_p,\Delta})}
\newcommand{\Spec}{\operatorname{Spec}}
\newcommand{\codim}{\operatorname{codim}}
\newcommand{\Zp}{\mathbb{Z}_p}
\newcommand{\Fp}{\mathbb{F}_p}
\newcommand{\Qp}{\mathbb{Q}_p}
\newcommand{\OK}{\mathcal{O}_K}
\newcommand{\OED}{\mathcal{O}_{\mathcal{E}_\Delta}}
\newcommand{\val}{\operatorname{val}}
\newcommand{\diag}{\operatorname{diag}}
\newcommand{\genlength}{{\operatorname{length}_{gen}}}
\newcommand{\length}{\operatorname{length}}

\newcommand{\Tr}{\operatorname{Tr}}
\newcommand{\Ext}{\operatorname{Ext}}
\newcommand{\Hom}{\operatorname{Hom}}

\newcommand{\Tor}{\operatorname{Tor}}
\newcommand{\Frac}{\operatorname{Frac}}

\newcommand{\GL}{\operatorname{GL}}
\newcommand{\Gal}{\operatorname{Gal}}

\newcommand{\id}{\operatorname{id}}

\newcommand{\rk}{\operatorname{rk}}

\newcommand{\Coker}{\operatorname{Coker}}

\newcommand{\bg}{(\hspace{-0.06cm}(}
\newcommand{\jg}{)\hspace{-0.06cm})}

\newcommand{\bs}{[\hspace{-0.04cm}[}
\newcommand{\js}{]\hspace{-0.04cm}]}
\newtheorem{thm}{Theorem}[section]

\newtheorem{pro}[thm]{Proposition}
\newtheorem{lem}[thm]{Lemma}
\newtheorem{cor}[thm]{Corollary}

\theoremstyle{definition}
\newtheorem*{rem}{Remark}
\newtheorem*{rems}{Remarks}

\begin{document}
\title{Multivariable $(\varphi,\Gamma)$-modules and products of Galois groups}
\author{Gergely Z\'abr\'adi \footnote{This research was supported by a Hungarian OTKA Research grant K-100291 and by the J\'anos Bolyai Scholarship of the Hungarian Academy of Sciences. I would like to thank the Arithmetic Geometry and Number Theory group of the University of Duisburg--Essen, campus Essen, for its hospitality and for financial support from SFB TR45 where parts of this paper was written.}}
\maketitle

\begin{abstract}
We show that the category of continuous representations of  the $d$th direct power of the absolute Galois group of $\mathbb{Q}_p$ on finite dimensional $\mathbb{F}_p$-vector spaces (resp.\ finitely generated $\Zp$-modules, resp.\ finite dimensional $\mathbb{Q}_p$-vector spaces) is equivalent to the category of \'etale $(\varphi,\Gamma)$-modules over a $d$-variable Laurent-series ring over $\mathbb{F}_p$ (resp.\ over $\Zp$, resp.\ over $\Qp$).
\end{abstract}

\section{Introduction}

This note serves as a complement to the work \cite{MultVar} where we relate multivariable $(\varphi,\Gamma)$-modules to smooth modulo $p^n$ representations of a split reductive group $G$ over $\mathbb{Q}_p$. The goal here is to show that the category of $d$-variable $(\varphi,\Gamma)$-modules is equivalent to the category of representations of the $d$th direct power of the absolute Galois group of $\Qp$. 

Let $K$ be a finite extension of $\Qp$ with ring of integers $\OK$, prime element $\varpi$, and residue field $\kappa$. For a finite set $\Delta$ let $\GQpD:=\prod_{\alpha\in\Delta}\Gal(\overline{\mathbb{Q}_p}/\mathbb{Q}_p)$ denote the direct power of the absolute Galois group of $\Qp$ indexed by $\Delta$. We denote by $\operatorname{Rep}_\kappa(\GQpD)$ (resp.\ by $\operatorname{Rep}_{\OK}(\GQpD)$, resp.\ by $\operatorname{Rep}_K(\GQpD)$) the category of continuous representations of the profinite group $\GQpD$ on finite dimensional $\kappa$-vector spaces (resp.\ finitely generated $\OK$-modules, resp.\ finite dimensional $K$-vector spaces). On the other hand, for independent commuting variables $X_\alpha$ ($\alpha\in\Delta$) we put 
\begin{eqnarray*}
E_{\Delta,\kappa}&:=&\kappa\bs X_\alpha\mid \alpha\in\Delta\js [X_\alpha^{-1}\mid \alpha\in\Delta]\ , \\
\mathcal{O}_{\mathcal{E}_{\Delta,K}}&:=&\varprojlim_h \left(\OK/\varpi^h\bs X_\alpha\mid \alpha\in\Delta\js [X_\alpha^{-1}\mid \alpha\in\Delta]\right)\ ,\\
\mathcal{E}_{\Delta,K}&:=&\mathcal{O}_{\mathcal{E}_{\Delta,K}}[p^{-1}]\ .
\end{eqnarray*}
Moreover, for each element $\alpha\in\Delta$ we have the partial Frobenius $\varphi_\alpha$, and group $\Gamma_\alpha\cong \Gal(\Qp(\mu_{p^\infty})/\Qp)$ acting on the variable $X_\alpha$ in the usual way and commuting with the other variables $X_\beta$ ($\beta\in\Delta\setminus\{\alpha\}$) in the above rings. A $(\varphi_\Delta,\Gamma_\Delta)$-module over $E_{\Delta,\kappa}$ (resp.\ over $\mathcal{O}_{\mathcal{E}_{\Delta,K}}$, resp.\ over $\mathcal{E}_{\Delta,K}$) is a finitely generated $E_{\Delta,\kappa}$-module (resp.\ $\mathcal{O}_{\mathcal{E}_{\Delta,K}}$-module, resp.\ $\mathcal{E}_{\Delta,K}$-module) $D$ together with commuting semilinear actions of the operators $\varphi_\alpha$ and groups $\Gamma_\alpha$ ($\alpha\in\Delta$). In case the coefficient ring is $E_{\Delta,\kappa}$ or $\mathcal{O}_{\mathcal{E}_{\Delta,K}}$, we say that $D$ is \'etale if the map $\id\otimes\varphi_\alpha\colon \varphi_\alpha^*D\to D$ is an isomorphism for all $\alpha\in\Delta$. For the coefficient ring $\mathcal{E}_{\Delta,K}$ we require the stronger assumption for the \'etale property that $D$ comes from an \'etale $(\varphi_\Delta,\Gamma_\Delta)$-module over $\mathcal{O}_{\mathcal{E}_{\Delta,K}}$ by inverting $p$. The main result of the paper is that $\operatorname{Rep}_\kappa(\GQpD)$ (resp.\ $\operatorname{Rep}_{\OK}(\GQpD)$, resp.\ $\operatorname{Rep}_K(\GQpD)$) is equivalent to the category of \'etale $(\varphi_\Delta,\Gamma_\Delta)$-modules over $E_{\Delta,\kappa}$ (resp.\ over $\mathcal{O}_{\mathcal{E}_{\Delta,K}}$, resp.\ over $\mathcal{E}_{\Delta,K}$).

Passing from the Galois side to $(\varphi_\Delta,\Gamma_\Delta)$-modules is rather straightforward. One constructs a big ring $E^{sep}_{\Delta}$ as an inductive limit of completed tensor products of finite separable extensions $E'_\alpha$ of $E_\alpha=\Fp\bg X_\alpha\jg$ ($\alpha\in\Delta$) over which the action of $\HQpD=\Ker(\GQpD\twoheadrightarrow\prod_{\alpha\in\Delta}\Gamma_\Delta)$ trivializes. The other direction is more involved. In order to trivialize the action of the partial Frobenii $\varphi_\alpha$ ($\alpha\in\Delta$) using induction, the main step is to find a lattice $D^{+*}_{\overline{\alpha}}$ integral in the variable $X_\alpha$ for some fixed $\alpha\in\Delta$ which is an \'etale $(\varphi_{\Delta\setminus\{\alpha\}},\Gamma_{\Delta\setminus\{\alpha\}})$-module over the ring $\Fp\bs X_\beta\mid \beta\in\Delta\js [X_\beta^{-1}\mid\beta\in\Delta\setminus\{\alpha\}]$. This uses the ideas of Colmez \cite{Mira} constructing lattices $D^+$ and $D^{++}$ in usual $(\varphi,\Gamma)$-modules.

We remark here that Scholze \cite{SchW} recently realized $\GQpD$ (using Drinfeld's Lemma for diamonds) as a geometric fundamental group $\pi_1((\mathrm{Spd}\ \Qp)^{|\Delta|}/\mathrm{p.Fr.})$ of the diamond $(\mathrm{Spd}\ \Qp)^{|\Delta|}$ modulo the partial Frobenii $\varphi_\beta$ ($\beta\in\Delta\setminus\{\alpha\}$) for some fixed $\alpha\in\Delta$: one can endow $E_\Delta^+=\Fp\bs X_\alpha\mid\alpha\in\Delta\js$ with its natural compact topology, and look at the subset of its adic spectrum $\operatorname{Spa}E_\Delta^+$ where all $X_\alpha$ ($\alpha\in\Delta$) are invertible. This defines an analytic adic space over $\Fp$, whose perfection modulo the
action of all $\Gamma_\alpha$'s is a model for $(\operatorname{Spd} \Qp)^d$. Thus, after taking the action modulo partial Frobenii $\varphi_\beta$ ($\beta\in\Delta\setminus\{\alpha\}$ for some fixed $\alpha\in\Delta$), the fundamental group will be $\GQpD$. Now, quite generally \'etale local systems on diamonds are equivalent to $\varphi$-modules. This introduces the last missing Frobenius, and one ends up with an equivalence between representations of $\GQpD$, and some sheaf of modules with $\Gamma_\Delta$-action and commuting actions of $\varphi_\alpha$ for all $\alpha\in\Delta$. However, this will not produce an actual module over a ring, but a sheaf of modules over a sheaf of rings. One can perhaps deduce the result of this paper along these lines, but that would require some further nontrivial input (replacing the above method of finding a lattice $D^{+*}_{\overline{\alpha}}$).

\subsection{Acknowledgements}

I would like to thank Christophe Breuil, Elmar Gro\ss e-Kl\"onne, Kiran Kedlaya, and Vytas Pa\v{s}k\={u}nas for useful discussions on the topic. I would like to thank Peter Scholze for clarifying the relation of this work to his theory of realizing $\GQpD$ as the \'etale fundamental group of a diamond.

\section{Algebraic properties of multivariable $(\varphi,\Gamma)$-modules}

\subsection{Definition and projectivity}

For a finite set $\Delta$ (which is the set of simple roots of $G$ in \cite{MultVar}) consider the Laurent series ring $E_{\Delta}:=E_{\Delta}^+[X_\Delta^{-1}]$ where $E_{\Delta}^+:=\mathbb{F}_p\bs X_\alpha\mid \alpha\in\Delta\js$ and $X_\Delta:=\prod_{\alpha\in\Delta}X_\alpha\in E_\Delta^+$. $E_\Delta^+$ is a regular noetherian local ring of global dimension $|\Delta|$, therefore $E_\Delta$ is a regular noetherian ring of global dimension $|\Delta|-1$. For each index $\alpha$ we define the action of the partial Frobenius $\varphi_\alpha$ and of the group $\Gamma_\alpha$ with $\chi_\alpha\colon\Gamma_\alpha\overset{\sim}{\to} \mathbb{Z}_p^\times$ on $E_\Delta$ as
\begin{eqnarray}
\varphi_\alpha(X_\beta)&:=&\begin{cases}X_\beta&\text{if }\beta\in\Delta\setminus\{\alpha\}\\ (X_\alpha+1)^p-1=X_\alpha^p&\text{if }\beta=\alpha\end{cases}\notag\\
\gamma_\alpha(X_\beta)&:=&\begin{cases}X_\beta&\text{if }\beta\in\Delta\setminus\{\alpha\}\\ (X_\alpha+1)^{\chi_\alpha(\gamma_\alpha)}-1&\text{if }\beta=\alpha\end{cases}\label{phigammaED}
\end{eqnarray}
for all $\gamma_\alpha\in\Gamma_\alpha$ extending the above formulas to continuous ring endomorphisms of $E_\Delta$ in the obvious way. By an \'etale $(\varphi_\Delta,\Gamma_\Delta)$-module over $E_\Delta $ we mean a (unless otherwise mentioned) finitely generated module $D$ over $E_\Delta $ together with a semilinear action of the (commutative) monoid $T_{+,\Delta}:=\prod_{\alpha\in\Delta}\varphi_\alpha^{\mathbb{N}}\Gamma_\alpha$ (also denote by $\varphi_t$ the action of $\varphi_t\in T_{+,\Delta}$ where the subscript $t$ is formal and refers to distinguishing between the elements of the set $T_{+,\Delta}$) such that the maps $$\id\otimes\varphi_t\colon \varphi_t^*D:=E_\Delta \otimes_{E_\Delta ,\varphi_t}D\to D$$ are isomorphisms for all elements $\varphi_t\in T_{+,\Delta}$. Here we put $\Gamma_\Delta:=\prod_{\alpha\in\Delta}\Gamma_\alpha$. We denote by $\mathcal{D}^{et}(\varphi_{\Delta},\Gamma_\Delta,E_\Delta)$ the category of \'etale $(\varphi_{\Delta},\Gamma_\Delta)$-modules over $E_\Delta$.

The category $\mathcal{D}^{et}(\varphi_{\Delta},\Gamma_\Delta,E_\Delta)$ has the structure of a neutral Tannakian category: For two objects $D_1$ and $D_2$ the tensor product $D_1\otimes_{E_\Delta}D_2$ is an \'etale $T_{+,\Delta}$-module with the action $\varphi_t(d_1\otimes d_2):=\varphi_t(d_1)\otimes\varphi_t(d_2)$ for $\varphi_t\in T_{+,\Delta,}$, $d_i\in D_i$ ($i=1,2$). Moreover, since $E_\Delta$ is a free module over itself via $\varphi_t$, putting $(\cdot)^*:=\Hom_{E_\Delta}(\cdot,E_\Delta)$ we have an identification $(\varphi_t^* D)^*\cong \varphi_t^*(D^*)$. So the isomorphism $\id\otimes\varphi_t\colon \varphi_t^*D\to D$ dualizes to an isomorphism $D^*\to \varphi_t^*(D^*)$. The inverse of this isomorphism (for all $\varphi_t\in T_{+,\Delta}$) equips $D^*$ with the structure of an \'etale $T_{+,\Delta}$-module.

\begin{lem}\label{cousingamma}
There exists a $\Gamma_\Delta$-equivariant injective resolution of $E_\Delta^+$ as a module over itself. 
\end{lem}
\begin{proof}
Consider the Cousin complex (see IV.2 in \cite{H})
\begin{align*}
0\to E_{\Delta}\to E_{\Delta,(0)}\to\dots\to\bigoplus_{\mathfrak{p}\in\Spec (E_\Delta), \codim\mathfrak{p}=r}J(\mathfrak{p})\to\dots
\end{align*}
where $J(\mathfrak{p})$ is the injective envelope of the residue field $\kappa(\mathfrak{p})$ as a module over the local ring $E_{\Delta,\mathfrak{p}}$. This is a $\Gamma_\Delta$-equivariant injective resolution since the action of $\Gamma_\Delta$ on $\Spec (E_\Delta)$ respects the codimension.
\end{proof}

\begin{pro}\label{proj}
Any object $D$ in $\mathcal{D}^{et}(\varphi_{\Delta},\Gamma_\Delta,E_\Delta)$ is a projective module over $E_\Delta$.
\end{pro}
\begin{proof}
Since $E_\Delta$ has finite global dimension, let $n$ be the projective dimension of $D$. Then by Lemma 4.1.6 in \cite{W} we have $\Ext^i(D,M)=0$ for all $i>n$ and $E_\Delta$-module $M$ and there exists an $R$-module $M_0$ with $\Ext^n(D,M_0)\neq 0$. By the long exact sequence of $\Ext$ and choosing an onto module homomorphism $F\twoheadrightarrow M_0$ from a free module $F$ we find that $\Ext^n(D,F)\neq 0$. Now $F$ is a (possibly infinite) direct sum of copies of $E_\Delta$ whence $\Ext^n(D,E_\Delta)\neq 0$ as $\Ext^n(D,\cdot)$ commutes with arbitrary direct sums. However, $\Ext^n(D,E_\Delta)$ is a finitely generated (as $E_\Delta$ is noetherian) torsion $E_\Delta$-module for $n>0$ (as all the modules in positive degrees in the Cousin complex above are torsion) admitting a semilinear action of $\Gamma_\Delta$ by Lemma \ref{cousingamma}. Therefore the global annihilator of $\Ext^n(D,E_\Delta)$ in $E_\Delta$ is a nonzero $\Gamma_\Delta$-invariant ideal in $E_\Delta$ hence equals $E_\Delta$ by Lemma 2.1 in \cite{MultVar}. So $n=0$ and $D$ is projective.
\end{proof}

\begin{lem}\label{stably}
We have $K_0(E_\Delta)\cong \mathbb{Z}$, ie.\ any finitely generated projective module over $E_\Delta$ is stably free.
\end{lem}
\begin{proof}
$E_\Delta^+\cong \mathbb{F}_p\bs X_\alpha\mid \alpha\in\Delta\js$ is a regular local ring, so it has finite global dimension and $K_0(E_\Delta^+)\cong G_0(E_\Delta^+)\cong\mathbb{Z}$ (Thm.\ II.7.8 in \cite{Kbook}). Therefore the localization $E_\Delta=E^+_\Delta[X_\Delta^{-1}]$ also has finite global dimension whence we have $K_0(E_\Delta)\cong G_0(E_\Delta)$. The statement follows noting that the map $G_0(E_\Delta^+)\to G_0(E_\Delta)$ is onto by the localization exact sequence of algebraic $K$-theory (Thm.\ II.6.4 in \cite{Kbook}).
\end{proof}

\begin{rem}
I am not aware of the analogue of the Theorem of Quillen and Suslin on the freeness of projective modules over $E_\Delta$. However, using the equivalence of categories of $\mathcal{D}^{et}(\varphi_{\Delta},\Gamma_\Delta,E_\Delta)$ with $\RepDFp$ we shall see later on (Cor.\ \ref{alwaysfree}) that any object $D$ in $\mathcal{D}^{et}(\varphi_{\Delta},\Gamma_\Delta,E_\Delta)$ is in fact free over $E_\Delta$.
\end{rem}

We equip $E_\Delta^+$ with the $X_\Delta$-adic topology. Then $(E_\Delta,E_\Delta^+)$ is a Huber pair (in the sense of \cite{SchW}) if we equip $E_\Delta$ with the inductive limit topology $E_\Delta=\bigcup_n X_\Delta^{-n}E_\Delta^+$. In fact, $E_\Delta$ is a complete noetherian Tate ring (op.\ cit.). Note that this is \emph{not} the natural compact topology on $E_\Delta^+$ as in the compact topology $E_\Delta^+$ would not be open in $E_\Delta$ since the index of $E_\Delta^+$ in $X_\Delta^{-n}E_{\Delta}^+$ is not finite. On the other hand, the inclusion $\mathbb{F}_p\bg X_\alpha\jg\hookrightarrow E_\Delta$ is \emph{not} continuous in the $X_\Delta$-adic topology (unless $|\Delta|=1$) therefore we cannot apply Drinfeld's Lemma (Thm.\ 17.2.4 in \cite{SchW}) directly in this situation.

Let $D$ be an object in $\mathcal{D}^{et}(\varphi_{\Delta},\Gamma_\Delta,E_\Delta)$. By Banach's Theorem for Tate rings (Prop.\ 6.18 in \cite{WH}), there is a unique $E_\Delta$-module topology on $D$ that we call the $X_\Delta$-adic topology. Moreover, any $E_\Delta$-module homomorphism is continuous in the $X_\Delta$-adic topology. 

\subsection{Integrality properties}

Put $\varphi_s:=\prod_{\alpha\in\Delta}\varphi_\alpha\in T_{+,\Delta}$ and define $D^{++}:=\{x\in D\mid \lim_{k\to\infty}\varphi_s^k(x)= 0\}$ where the limit is considered in the $X_\Delta$-adic topology (cf.\ II.2.1 in \cite{Mira} in case $|\Delta|=1$). Note that $\varphi_s$ is the absolute Frobenius on $E_\Delta$, it takes any element to its $p$th power.

\begin{lem}\label{phistarfingen}
Let $M$ be a finitely generated $E_\Delta^+$-submodule in $D$. Then $E_\Delta^+\varphi_s(M)$ is also finitely generated.
\end{lem}
\begin{proof}
If $M$ is generated by $m_1,\dots,m_n$ then $\varphi_s(m_1),\dots,\varphi_s(m_n)$ generate $E_\Delta^+\varphi_s(M)$.
\end{proof}

\begin{pro}
$D^{++}$ is a finitely generated $E_\Delta^+$-submodule in $D$ that is stable under the action of $T_{+,\Delta}$ and we have $D=D^{++}[X_\Delta^{-1}]$.
\end{pro}
\begin{proof}
Choose an arbitrary finitely generated $E_\Delta^+$-submodule $M$ of $D$ with $M[X_\Delta^{-1}]=D$ (e.g.\ take $M=E_\Delta^+e_1+\dots+E_\Delta^+ e_n$ for some $E_\Delta$-generating system $e_1,\dots,e_n$ of $D$). By Lemma \ref{phistarfingen} we have an integer $r\geq 0$ such that $\varphi_s(M)\subseteq X_\Delta^{-r}M$, since $E_\Delta^+$ is noetherian and we have $D=\bigcup_r X_\Delta^{-r}M$. Then we have $$\varphi_s(X_\Delta^k M)=X_\Delta^{pk}\varphi_s(M)\subseteq X_\Delta^{pk-r}M\subseteq X_\Delta^{k+1}M$$
for any integer $k\geq \frac{r+1}{p-1}$. Therefore we have $X_\Delta^{\left[\frac{r+1}{p-1}\right]+1}M\subseteq D^{++}$ whence $D^{++}[X_\Delta^{-1}]=M[X_\Delta^{-1}]=D$.

Since $T_{+\Delta}$ is commutative and the action of each $\varphi_t\in T_{+,\Delta}$ is continuous, $D^{++}$ is stable under the action of $T_{+,\Delta}$. There is a system of neighbourhoods of $0$ in $D$ consisting of $E_\Delta^+$-submodules therefore $D^{++}$ is an $E_\Delta^+$-submodule.

To prove that $D^{++}$ is finitely generated over $E^+_\Delta$ suppose first that $D$ is a free module over $E_\Delta$ generated by $e_1,\dots,e_n$ and put $M:=E^+_\Delta e_1+\dots + E_\Delta^+ e_n$. We may assume $M\subseteq D^{++}$ by replacing $M$ with $X_\Delta^{\left[\frac{r+1}{p-1}\right]+1}M$. Moreover, further multiplying $M=E^+_\Delta e_1+\dots +E^+_\Delta e_n$ by a power of $X_\Delta$, we may assume that the matrix $A:=[\varphi_s]_{e_1,\dots,e_n}$ of $\varphi_s$ in the basis $e_1,\dots,e_n$ lies in ${E^+_\Delta}^{n\times n}$ as we have $[\varphi_s]_{X_\Delta^re_1,\dots,X_\Delta^re_n}=X_\Delta^{(p-1)r}[\varphi_s]_{e_1,\dots,e_n}$. Now we choose the integer $r>0$ so that it is bigger than $\val_{X_\alpha}(\det A)$ for all $\alpha\in\Delta$ and claim that $D^{++}\subseteq X_\Delta^{-r}M$ whence $D^{++}$ is finitely generated over $E^+_\Delta$ as $E^+_\Delta$ is noetherian. Assume for contradiction that $d=\sum_{i=1}^nd_ie_i$ lies in $D^{++}$ for some $d_i\in E_\Delta$ ($i=1,\dots,n$) such that at least one $d_i$, say $d_1$, does not lie in $X_\Delta^{-r}E^+_\Delta$. In particular, there exists an $\alpha$ in $\Delta$ such that $\val_{X_\alpha}(d_1)<-r$. Since $M$ is open in $D$ and $d\in D^{++}$, there exists an integer $k>0$ such that $\varphi_s^k(d)$ is in $M$ which is equivalent to saying that the column vector
\begin{equation*}
A\varphi_s(A)\dots\varphi_s^{k-1}(A)\begin{pmatrix}\varphi_s^k(d_1)\\ \vdots \\ \varphi_s^k(d_n)\end{pmatrix}
\end{equation*}
lies in ${E^+_\Delta}^n$. Multiplying this by the matrix built from the $(n-1)\times(n-1)$ minors of $A\varphi_s(A)\dots\varphi_s^{k-1}(A)$ we deduce that $\det(A\varphi_s(A)\dots\varphi_s^{k-1}(A))\varphi_s^k(d_1)=\det(A)^{\frac{p^k-1}{p-1}}d_1^{p^k}$ lies in $E^+_\Delta$. We compute
\begin{align*}
0\leq \val_{X_\alpha}(\det(A)^{\frac{p^k-1}{p-1}}d_1^{p^k})=\frac{p^k-1}{p-1}\val_{X_\alpha}(\det(A))+p^k\val_{X_\alpha}(d_1)<\\
<\frac{p^k-1}{p-1}\val_{X_\alpha}(\det(A))-p^kr<0
\end{align*}
by our assumption that $r>\val_{X_\alpha}(\det(A))$, yielding a contradiction.

In the general case note that $D$ is always stably free by Prop.\ \ref{proj} and Lemma \ref{stably}. So $D_1:=D\oplus E_\Delta^k$ is a free module over $E_\Delta$ for $k$ large enough. We make $D_1$ into an \'etale $T_{+,\Delta}$-module by the trivial action of $T_{+,\Delta}$ on $E_\Delta^k$ to deduce that $D_1^{++}$ is finitely generated over $E_\Delta^+$. The result follows noting that $D^{++}\subseteq D_1^{++}$ and $E_\Delta^+$ is noetherian. 
\end{proof}

For an object $D$ in $\mathcal{D}^{et}(\varphi_{\Delta},\Gamma_\Delta,E_\Delta)$ we define $$D^+:=\{x\in D\mid \{\varphi_s^k(x)\colon k\geq 0\}\subset D\text{ is bounded}\}\ .$$ Since $\varphi_s^k(X_\Delta)$ tends to $0$ in the $X_\Delta$-adic topology, we have $X_\Delta D^+\subseteq D^{++}$, ie.\ $D^+\subseteq X_\Delta^{-1}D^{++}$. In particular, $D^+$ is finitely generated over $E_\Delta^+$. On the other hand, we also have $D^{++}\subseteq D^+$ by construction whence we deduce $D=D^+[X_\Delta^{-1}]$.

\begin{lem}\label{tstableD+}
We have $\varphi_t(D^+)\subset D^+$ (resp.\ $\varphi_t(D^{++})\subset D^{++}$) for all $\varphi_t\in T_{+,\Delta}$.
\end{lem}
\begin{proof}
For any generating system $e_1,\dots,e_n$ of $D$ and any $\varphi_t\in T_{+,\Delta}$ there exists an integer $k=k(\varphi_t,M)>0$ such that we have $\varphi_t(X_\Delta^kM)\subseteq X_\Delta^kE_\Delta^+\varphi_t(M)\subseteq M$ where we put $M:=E_\Delta^+e_1+\dots +E_\Delta^+e_n$ by Lemma \ref{phistarfingen}. Indeed, $X_\Delta$ divides $\varphi_t(X_\Delta)$ in $E^+_\Delta$, and we have $D=M[X_\Delta^{-1}]$ by construction. The statement on $D^{++}$ follows from the commutativity of the monoid $T_{+,\Delta}$ noting that there exists a basis of neighbouhoods of $0$ in $D$ consisting of $E_\Delta^+$-submodules of the form $M$. To see that $\varphi_t(D^+)\subseteq D^+$ note that $\varphi_t(D^+)$ is bounded and we have $\varphi_s^k(\varphi_t(D^+))=\varphi_t(\varphi_s^k(D^+))\subset \varphi_t(D^+)$.
\end{proof}

Now fix an $\alpha\in \Delta$ and define $D^+_{\overline{\alpha}}:=D^+[X_{\Delta\setminus\{\alpha\}}^{-1}]$ where for any subset $S\subseteq \Delta$ we put $X_S:=\prod_{\beta\in S}X_\beta$. Then $D^+_{\overline{\alpha}}$ is a finitely generated module over $E^+_{\overline{\alpha}}:=E^+_\Delta[X_{\Delta\setminus\{\alpha\}}^{-1}]$. We denote by $T_{+,\overline{\alpha}}\subset T_{+,\Delta}$ the monoid generated by $\varphi_\beta$ ($\beta\in \Delta\setminus\{\alpha\}$) and $\Gamma_{\Delta}$. 

\begin{lem}\label{twoprodinD+}
$D^+_{\overline{\alpha}}/D^+$ is $X_\alpha$-torsion free: If both $X_\alpha^{n_1} d$ and $X_{\Delta\setminus\{\alpha\}}^{n_2}d$ lie in $D^+$ for some element $d\in D$, $\alpha\in\Delta$, and integers $n_1,n_2\geq 0$ then we have $d\in D^+$. The same statement holds if we replace $D^+$ by $D^{++}$.
\end{lem}
\begin{proof}
At first assume that $D$ is free as a module over $E_\Delta$ with basis $e_1,\dots,e_n$. Then the denominators of $\varphi_s^k(X_\alpha^{n_1} d)=X_\alpha^{n_1p^k}\varphi_s^k(d)$ in the basis $e_1,\dots,e_n$ are bounded for $k\geq 0$ by assumption. Therefore the $X_\beta$-valuations of the denominators of $\varphi_s^k(d)$ are bounded for all $\beta\in\Delta\setminus\{\alpha\}$ since $E_\Delta^+$ is a unique factorization domain. On the other hand, the $X_\alpha$-valuations of these denominators are also bounded since the denominators of $\varphi_s^k(X_{\Delta\setminus\{\alpha\}}^{n_2}d)=X_{\Delta\setminus\{\alpha\}}^{n_2p^k}\varphi_s^k(d)$ are bounded. To prove the statement for $D^{++}$ we have the same argument but `being bounded' replaced by `tends to $0$'.

Finally, by Prop.\ \ref{proj} and Lemma \ref{stably} $D\oplus E_\Delta^k$ is free over $E_\Delta$ and we equip it with the structure of an \'etale $(\varphi,\Gamma)$-module (trivially on $E_\Delta^k$). The statement follows from the additivity of the constructions $D\mapsto D^+$ and $D\mapsto D^+_{\overline{\alpha}}$ in direct sums.
\end{proof}

\begin{lem}\label{rank1val}
Assume that $D$ is generated by a single element $e_1\in D$ over $E_\Delta$. Then for any $\varphi_t$ in $T_{+,\overline{\alpha}}$ we have $\varphi_t(e_1)=a_te_1$ for some unit $a_t$ in $(E^+_{\overline{\alpha}})^\times$.
\end{lem}
\begin{proof}
Define $a_t\in E_\Delta$ and $a_\alpha\in E_\Delta$ so that $\varphi_t(e_1)=a_te_1$ and $\varphi_\alpha(e_1)=a_\alpha e_1$. By the \'etale property both $a_t$ and $a_\alpha$ are units in $E_\Delta$, so it remains to show that $\val_{X_\alpha}(a_t)=0$. We compute
\begin{align*}
\varphi_\alpha(a_t)a_\alpha e_1=\varphi_\alpha(a_t)\varphi_\alpha(e_1)=\varphi_\alpha(a_t e_1)=\varphi_\alpha(\varphi_t(e_1))=\\
=\varphi_t(\varphi_\alpha(e_1))=\varphi_t(a_\alpha e_1)=\varphi_t(a_\alpha)\varphi_t(e_1)=\varphi_t(a_\alpha)a_te_1
\end{align*}
whence we deduce
\begin{align*}
p\val_{X_\alpha}(a_t)+\val_{X_\alpha}(a_\alpha)=\val_{X_\alpha}(\varphi_\alpha(a_t)a_\alpha)=\val_{X_\alpha}(\varphi_t(a_\alpha)a_t)=\val_{X_\alpha}(a_\alpha)+\val_{X_\alpha}(a_t)\ .
\end{align*}
This yields $\val_{X_\alpha}(a_t)=0$ as required.
\end{proof}

\begin{lem}\label{intersectnon0}
There exists an integer $k=k(D)>0$ such that for any $\varphi_t\in T_{+,\overline{\alpha}}$ we have $X_\alpha^kD^+_{\overline{\alpha}}\subseteq E_\Delta^+\varphi_t(D^+_{\overline{\alpha}})$.
\end{lem}
\begin{proof}
At first assume that $D$ is free, choose a basis $e_1,\dots,e_n$ contained in $D^{+}$, and put $M:=E_\Delta^+e_1+\dots E_\Delta^+e_n$, $M_\alpha:=E_{\overline{\alpha}}^+e_1+\dots +E_{\overline{\alpha}}^+e_n$. There exists an integer $k_0>0$ such that $D^+\subseteq X_\Delta^{-k_0}M$. In particular, we have $D^+_{\overline{\alpha}}\subseteq X_\alpha^{-k_0}M_{\overline{\alpha}}$. Now for a fixed $\varphi_t\in T_{+,\overline{\alpha}}$ let $A_t\in E_\Delta^{n\times n}$ be the matrix of $\varphi_t$ in the basis $e_1,\dots,e_n$. Since $\varphi_t(e_i)$ lies in $D^+\subseteq X_\alpha^{-k_0}M_{\overline{\alpha}}$, all the entries of the matrix $A_t$ are in $X_\alpha^{-k_0}E^+_{\overline{\alpha}}$. Applying Lemma \ref{rank1val} to the single generator $e_1\wedge\dots\wedge e_n$ of $\bigwedge^n D$ we obtain $\val_{X_\alpha}(\det A_t)=0$. In particular, all the entries of $A_t^{-1}$ lie in $X_\alpha^{-(n-1)k_0}E^+_{\overline{\alpha}}$ by the formula for the inverse matrix using the $(n-1)\times(n-1)$ minors in $A_t$. Now note that the elements $e_1,\dots,e_n$ can be written as a linear combination of $\varphi_t(e_1),\dots,\varphi_t(e_n)$ with coefficients from $A_t^{-1}$. Using Lemma \ref{tstableD+} this shows 
\begin{equation*}
X_\alpha^{k_0}D^+_{\overline{\alpha}}\subseteq M_{\overline{\alpha}}\subseteq X_\alpha^{-(n-1)k_0}\varphi_t(M_{\overline{\alpha}})\subseteq X_\alpha^{-(n-1)k_0}\varphi_t(D^+_{\overline{\alpha}})\ .
\end{equation*}
So we may choose $k:=nk_0$ independent of $\varphi_t$.

The general case follows from Prop.\ \ref{proj} and Lemma \ref{stably} noting that the functor $D\mapsto D^+_{\overline{\alpha}}$ commutes with direct sums.
\end{proof}

In view of the above Lemma we define $$D^{+\ast}_{\overline{\alpha}}:=\bigcap_{\varphi_t\in T_{+,\overline{\alpha}}}E^+_{\overline{\alpha}}\varphi_t(D^+_{\overline{\alpha}})\ .$$ $D^{+\ast}_{\overline{\alpha}}$ is finitely generated over $E^+_{\overline{\alpha}}$ as it is contained in $D^+_{\overline{\alpha}}$ and $E^+_{\overline{\alpha}}$ is noetherian. On the other hand, by Lemma \ref{intersectnon0} we have $X_\alpha^k D^+_{\overline{\alpha}}\subseteq D^{+\ast}_{\overline{\alpha}}$ for some integer $k=k(D)>0$ whence, in particular, $D=D^{+\ast}_{\overline{\alpha}}[X_\alpha^{-1}]$.

\begin{pro}\label{alphaintetale}
$D^{+\ast}_{\overline{\alpha}}$ is an \'etale $T_{+,\overline{\alpha}}$-module over $E^+_{\overline{\alpha}}$, ie.\ the maps 
\begin{equation}
\id\otimes\varphi_t\colon \varphi_t^*D^{+\ast}_{\overline{\alpha}}=E^+_{\overline{\alpha}}\otimes_{E^+_{\overline{\alpha}},\varphi_t}D^{+\ast}_{\overline{\alpha}}\to D^{+\ast}_{\overline{\alpha}} \label{varphiastDalpha}
\end{equation}
are bijective for all $\varphi_t\in T_{+,\alpha}$.
\end{pro}
\begin{proof}
At first note that we have $\varphi_t(D^{+\ast}_{\overline{\alpha}})\subseteq D^{+\ast}_{\overline{\alpha}}$ for all $\varphi_t\in T_{+,\overline{\alpha}}$ by Lemma \ref{tstableD+} and the commutativity of $T_{+,\overline{\alpha}}$, so the map \eqref{varphiastDalpha} exists. Now let $\varphi_{t_0}\in T_{+,\overline{\alpha}}$ be arbitrary. Since $E^+_{\overline{\alpha}}$ (resp.\ $E_\Delta$) is a finite free module over $\varphi_{t_0}(E^+_{\overline{\alpha}})$ (resp.\ over $\varphi_{t_0}(E_\Delta)$) with generators contained in $E^+_\Delta$, we have a natural identification $\varphi_{t_0}^*D^{+\ast}_{\overline{\alpha}}\cong E^+_{\Delta}\otimes_{E^+_{\Delta},\varphi_{t_0}}D^{+\ast}_{\Delta}$ (resp.\ $\varphi_{t_0}^*D\cong E^+_\Delta \otimes_{E^+_{\Delta},\varphi_{t_0}}D$). Since $E^+_\Delta$ is finite free (hence flat) over $\varphi_{t_0}(E^+_\Delta)$, the inclusion $D^+_{\overline{\alpha}}\subset D$ induces an inclusion $\varphi_{t_0}^*D^+_{\overline{\alpha}}\subset \varphi_{t_0}^*D$. It follows that \eqref{varphiastDalpha} is injective since $D$ is \'etale. Similarly, for each $\varphi_t\in T_{+,\overline{\alpha}}$, the map $$\id\otimes\varphi_{t_0}\colon \varphi_{t_0}^*(E^+_{\overline{\alpha}}\varphi_t(D^+_{\overline{\alpha}}))\to E^+_{\overline{\alpha}}\varphi_t(D^+_{\overline{\alpha}})$$ is injective with image $E^+_{\overline{\alpha}}\varphi_{t_0}\varphi_t(D^+_{\overline{\alpha}})$. On the other hand, since $E^+_\Delta$ is finite free over $\varphi_{t_0}(E^+_\Delta)$, we have $\varphi_{t_0}^*D^{+\ast}_{\overline{\alpha}}=\bigcap_{t\in T_{+,\overline{\alpha}}}\varphi_{t_0}^*(E^+_{\overline{\alpha}}\varphi_t(D^+_{\overline{\alpha}}))$ where the intersection is taken inside $\varphi_{t_0}^*D$. Therefore \eqref{varphiastDalpha} is bijective as we have $D^{+\ast}_{\overline{\alpha}}=\bigcap_{\varphi_t\in T_{+,\overline{\alpha}}}E^+_{\overline{\alpha}}\varphi_{t_0}\varphi_t(D^+_{\overline{\alpha}})$.
\end{proof}

\begin{lem}\label{phicover}
There exists a finitely generated $E_\Delta^+$-submodule $D_0\subset D^{+*}_{\overline{\alpha}}$ such that $D_0\subseteq E_\Delta^+\varphi_{\overline{\alpha}}(D_0)$ and $D^{+*}_{\overline{\alpha}}=D_0[X_{\Delta\setminus\{\alpha\}}^{-1}]$ where $\varphi_{\overline{\alpha}}:=\prod_{\beta\in\Delta\setminus\{\alpha\}}\varphi_\beta$. Moreover, we have $D^{+*}_{\overline{\alpha}}=\bigcup_{r\geq 0}E^+_\Delta\varphi_{\overline{\alpha}}^r(X_{\Delta\setminus\{\alpha\}}^{-1}D_0)$.
\end{lem}
\begin{proof}
Put $D_1:=D^+\cap D^{+*}_{\overline{\alpha}}$. By Prop.\ \ref{alphaintetale} and the fact that $D^{+*}_{\overline{\alpha}}=D_1[X_{\Delta\setminus\{\alpha\}}^{-1}]$ we find an integer $k_0>0$ such that $X_{\Delta\setminus\{\alpha\}}^{k_0}D_1\subseteq E^+_\Delta\varphi_{\overline{\alpha}}(D_1)$. So for $k>\frac{k_0}{p-1}$ we have $$X_{\Delta\setminus\{\alpha\}}^{-k}D_1\subseteq X_{\Delta\setminus\{\alpha\}}^{-k-k_0}E^+_\Delta\varphi_{\overline{\alpha}}(D_1) \subseteq X_{\Delta\setminus\{\alpha\}}^{-pk}E^+_\Delta\varphi_{\overline{\alpha}}(D_1)=E^+_\Delta\varphi_{\overline{\alpha}}(X_{\Delta\setminus\{\alpha\}}^{-k}D_1)\ .$$ So we put $D_0:=X_{\Delta\setminus\{\alpha\}}^{-k}D_1$ so that the first part of the statement is satisfied. Iterating the inclusion $D_0\subseteq E^+_\Delta\varphi_{\overline{\alpha}}(D_0)$ we obtain $D_0\subseteq E^+_\Delta\varphi_{\overline{\alpha}}^r(D_0)$ for all $r\geq 1$. Finally, we compute $$X_{\Delta\setminus\{\alpha\}}^{-p^r}D_0\subseteq X_{\Delta\setminus\{\alpha\}}^{-p^r} E^+_\Delta\varphi_{\overline{\alpha}}^r(D_0)= E^+_\Delta\varphi_{\overline{\alpha}}^r(X_{\Delta\setminus\{\alpha\}}^{-1}D_0)\ .$$
The statement follows noting that we have $D^{+*}_{\overline{\alpha}}=D_0[X_{\Delta\setminus\{\alpha\}}^{-1}]=\bigcup_r X_{\Delta\setminus\{\alpha\}}^{-p^r}D_0$.
\end{proof}

\section{The equivalence of categories for $\mathbb{F}_p$-representations}

\subsection{The functor $\mathbb{D}$}

Take a copy $G_{\mathbb{Q}_p,\alpha}\cong \Gal(\overline{\mathbb{Q}_p}/\mathbb{Q}_p)$ of the absolute Galois group of $\mathbb{Q}_p$ for each element $\alpha\in\Delta$ and let $G_{\mathbb{Q}_p,\Delta}:=\prod_{\alpha\in\Delta}G_{\mathbb{Q}_p,\alpha}$. Let $\RepDFp$ be the category of continuous representations of the group $\GQpD$ on finite dimensional $\mathbb{F}_p$ vectorspaces. We identify $\Gamma_\alpha$ with the Galois group $\Gal(\mathbb{Q}_p(\mu_{p^\infty})/\mathbb{Q}_p)$ as a quotient of $\GQpa$ via the cyclotomic character $\chi_\alpha\colon \Gal(\mathbb{Q}_p(\mu_{p^\infty})/\mathbb{Q}_p)\to \Zp^\times$. Further, we denote by $\HQpa$ the kernel of the natural quotient map $\GQpa\to \Gamma_\alpha$ and put $\HQpD:=\prod_{\alpha\in\Delta}\HQpa\lhd \GQpD$. Putting $E_\alpha:=\mathbb{F}_p\bg X_\alpha\jg$ we have the following fundamental result of Fontaine and Wintenberger (Thm.\ 4.16 \cite{FO}).
\begin{thm}
The absolute Galois group $\Gal(E_\alpha^{sep}/E_\alpha)$ is isomorphic to $\HQpa$. Moreover, $\GQpa$ acts on the separable closure $E_\alpha^{sep}$ via automorphisms such that the action of $\Gamma_\alpha\cong \GQpa/\HQpa$ on $E_\alpha=(E_\alpha^{sep})^{\HQpa}$ coincides with the one given in \eqref{phigammaED}.
\end{thm}

For each $\alpha\in\Delta$ consider a finite separable extension $E'_\alpha$ of $E_\alpha$ together with the Frobenius $\varphi_\alpha\colon E'_\alpha\to E'_\alpha$ acting by raising to the power $p$. We denote by $E'^+_\alpha$ the integral closure of $E_\alpha^+=\mathbb{F}_p\bs X_\alpha\js$ in $E'_\alpha$. Note that $E'_\alpha$ is isomorphic to $\mathbb{F}_{q_\alpha}\bg X'_\alpha\jg$ for some power $q_\alpha$ of $p$ and uniformizer $X'_\alpha$ such that we have $E'^+_\alpha\cong \mathbb{F}_{q_\alpha}\bs X'_\alpha\js$. We normalize the $X_\alpha$-adic (multiplicative) valuation on $E_\alpha$ so that we have $|X_\alpha|_{X_\alpha}=p^{-1}$. This extends uniquely to the finite extension $E'_\alpha$. Moreover, we equip the tensor product $E'_{\Delta,\circ}:=\bigotimes_{\alpha\in\Delta,\mathbb{F}_p}E'_\alpha$ with a norm $|\cdot|_{prod}$ by the formula 
\begin{equation}\label{prodnorm}
|c|_{prod}:=\inf\left(\max_i(\prod_{\alpha\in\Delta}|c_{\alpha,i}|_\alpha)\ \mid\ c=\sum_{i=1}^n\bigotimes_{\alpha\in\Delta}c_{\alpha,i}\right)\ .
\end{equation}
Note that the restriction of $|\cdot|_{prod}$ to the subring $E'^+_{\Delta,\circ}:=\bigotimes_{\alpha\in\Delta,\mathbb{F}_p}E'^+_\alpha$ induces the valuation with respect to the augmentation ideal $\Ker(E'^+_{\Delta,\circ}\twoheadrightarrow \bigotimes_{\alpha\in\Delta,\mathbb{F}_p}\mathbb{F}_{q_\alpha})$. The norm $|\cdot|_{prod}$ is not multiplicative in general, as the ring $\bigotimes_{\alpha\in\Delta,\mathbb{F}_p}\mathbb{F}_{q_\alpha}$ is not a domain. However, it is submultiplicative. We define $E'^+_\Delta$ as the completion of $E'^+_{\Delta,\circ}$ with respect to $|\cdot|_{prod}$ and put $E'_\Delta:=E'^+_\Delta[1/X_\Delta]$. Note that $E'_\Delta$ is \emph{not} complete with respect to $|\cdot|_{prod}$ (unless $|\Delta|=1$) even though $E'_{\Delta,\circ}=E'^+_{\Delta,\circ}[1/X_\Delta]$ is a dense subring in $E'_\Delta$. Since we have a containment 
\begin{equation*}
(\bigotimes_{\alpha\in\Delta,\mathbb{F}_p}\mathbb{F}_{q_\alpha})[X'_\alpha,\alpha\in \Delta]=\bigotimes_{\alpha\in\Delta,\mathbb{F}_p}\mathbb{F}_{q_\alpha}[X'_\alpha]\leq_{dense} E'^+_{\Delta,\circ}
\end{equation*}
we may identify $E'^+_\Delta$ with the power series ring $(\bigotimes_{\alpha\in\Delta,\mathbb{F}_p}\mathbb{F}_{q_\alpha})\bs X'_\alpha,\alpha\in \Delta\js$ which is the completion of the polynomial ring above. In particular, the special case $E'_\alpha=E_\alpha$ for all $\alpha\in\Delta$ yields a ring $E'_\Delta$ isomorphic to $E_\Delta$. Therefore $E_\Delta$ is a subring of $E'_\Delta$ for all collections of finite separable extensions $E'_\alpha$ of $E_\alpha$ ($\alpha\in\Delta$). Further, $\varphi_\alpha$ acts on $E'^+_{\Delta,\circ}$ (and on $E'_{\Delta,\circ}$) by the Frobenius on the component in $E'_\alpha$ and by the identity on all the other components in $E'_\beta$, $\beta\in\Delta\setminus\{\alpha\}$. This action is continuous in the norm $|\cdot|_{prod}$ therefore extends to the completion $E'^+_\Delta$ and the localization $E'_\Delta$. We have the following alternative characterization of the ring $E'_\Delta$.

\begin{lem}\label{E'Dfingen}
Put $\Delta=\{\alpha_1,\dots,\alpha_n\}$. We have
\begin{equation*}
E'_\Delta\cong E'_{\alpha_1}\otimes_{E_{\alpha_1}}\left(E'_{\alpha_2}\otimes_{E_{\alpha_2}}\left(\cdots(E'_{\alpha_n}\otimes_{E_{\alpha_n}}E_\Delta)\right)\right)\ .
\end{equation*}
\end{lem}
\begin{proof}
By rearranging the order of tensor products we have an identification
\begin{equation*}
E'^+_{\Delta,\circ}=\bigotimes_{\alpha\in\Delta,\mathbb{F}_p}(E'^+_\alpha\otimes_{E^+_\alpha}E^+_\alpha)\cong E'^+_{\alpha_1}\otimes_{E^+_{\alpha_1}}\left(E'^+_{\alpha_2}\otimes_{E^+_{\alpha_2}}\left(\dots(E'^+_{\alpha_n}\otimes_{E^+_{\alpha_n}}E^+_{\Delta,\circ})\right)\right)\ ,
\end{equation*}
where $E^+_{\Delta,\circ}$ is just $E'^+_{\Delta,\circ}$ with the choice $E'_\alpha=E_\alpha$ for all $\alpha\in\Delta$. The statement follows by completing this with respect to the maximal ideal of $E^+_{\Delta,\circ}$ and inverting $X_\Delta$.
\end{proof}

We define the multivariable analogue of $E^{sep}$ as
\begin{equation*}
E^{sep}_\Delta:=\varinjlim_{E_\alpha\leq E'_\alpha\leq E^{sep}_\alpha,\forall\alpha\in\Delta}E'_\Delta\ .
\end{equation*}

For any subset $S\subseteq \Delta$ we define the similar notions $E'^+_S$, $E'_S$, and $E^{sep}_S$ with $\Delta$ replaced by $S$. We equip $E^{sep}_\Delta$ with the relative Frobenii $\varphi_\alpha$ for each $\alpha\in\Delta$ defined above on each $E'_\Delta$. Further, $E^{sep}_\Delta$ admits an action of $\GQpD$ satisfying

\begin{pro}\label{invEsepD}
Assume that the extensions $E'_\alpha/E_\alpha$ are Galois for all $\alpha\in\Delta$ and let $H':=\prod_{\alpha\in\Delta}H'_\alpha$ where $H'_\alpha:=\Gal(E^{sep}_\alpha/E'_\alpha)$. Then we have $(E^{sep}_\Delta)^{H'_\Delta}=E'_\Delta$. In particular, the subring $(E^{sep}_\Delta)^{\HQpD}$ of $\HQpD$-invariants in $E^{sep}_\Delta$ equals $E_\Delta$ with the previously defined action of $\Gamma_\Delta\cong \GQpD/\HQpD$.
\end{pro}
\begin{proof}
Since $X_\Delta$ is $H'_\Delta$-invariant and $\varinjlim$ can be interchanged with taking $ H'_\Delta$-invariants, it suffices to show that whenever $$E_\alpha=\mathbb{F}_p\bg X_\alpha\jg\leq E'_\alpha=\mathbb{F}_{q'_\alpha}\bg X'_\alpha\jg\leq E''_\alpha=\mathbb{F}_{q''_\alpha}\bg X''_\alpha\jg$$ is a sequence of finite Galois extensions for each $\alpha\in\Delta$ then we have $(E''^+_\Delta)^{ H'_\Delta}=E'^+_\Delta$. The containment $(E''^+_\Delta)^{ H'_\Delta}\supseteq E'^+_\Delta$ is clear. We prove the converse by induction on $|\Delta|$. Note that the ideal $\mathcal{M}_{\alpha}\lhd E''^+_\Delta$ generated by $X''_\alpha$ is invariant under the action of $ H'_\Delta$ for any fixed $\alpha$ in $\Delta$. Moreover, for any integer $k\geq 1$ the ring $E''^+_{\alpha}/\mathcal{M}_{\alpha}^k$ is finite dimensional over $\mathbb{F}_p$. Therefore the image of $(E''^+_\Delta)^{ H'_\Delta}$ under the quotient map $E''^+_\Delta\twoheadrightarrow E''^+_\Delta/\mathcal{M}_{\alpha}^k$ is contained in
\begin{align*}
\left(E''^+_\Delta/\mathcal{M}_{\alpha}^k\right)^{H'_\Delta}\subseteq \left(E''^+_\Delta/\mathcal{M}_{\alpha}^k\right)^{ H'_{\Delta\setminus\{\alpha\}}}= \left(E''^+_{\Delta\setminus\{\alpha\}}\otimes_{\mathbb{F}_p}\left(E''^+_{\alpha}/\mathcal{M}_{\alpha}^k\right)\right)^{ H'_{\Delta\setminus\{\alpha\}}}=\\
=\left(E''^+_{\Delta\setminus\{\alpha\}}\right)^{ H'_{\Delta\setminus\{\alpha\}}}\otimes_{\mathbb{F}_p}\left(E''^+_{\alpha}/\mathcal{M}_{\alpha}^k\right)=
E'^+_{\Delta\setminus\{\alpha\}}\otimes_{\mathbb{F}_p}\left(E''^+_{\alpha}/\mathcal{M}_{\alpha}^k\right)
\end{align*}
by induction. Taking the projective limit with respect to $k\geq 1$ we deduce that $(E''^+_\Delta)^{H'_\Delta}$ is contained in the power series ring $$\left(\mathbb{F}_{q''_\alpha}\otimes_{\mathbb{F}_p}\bigotimes_{\beta\in\Delta\setminus\{\alpha\},\mathbb{F}_p}\mathbb{F}_{q'_\beta}\right)\bs X''_\alpha,X'_\beta\mid \beta\in\Delta\setminus\{\alpha\}\js\subseteq E''^+_\Delta\ .$$ Now using the action of $H'_\alpha$ in a similar argument as above (reducing modulo the $k$th power of the ideal generated by all the $X'_\beta$, $\beta\in\Delta\setminus\{\alpha\}$ for all $k\geq 1$) we deduce the statement.
\end{proof}

The subring $E^{sep}_{\Delta,\circ}\cong \bigotimes_{\alpha\in\Delta,\mathbb{F}_p}E^{sep}_\alpha$ in $E^{sep}_\Delta$ is the inductive limit of $E'_{\Delta,\circ}\subseteq E'_{\Delta}$ where $E'_\alpha$ runs through the finite separable extensions of $E_\alpha$ for each $\alpha\in\Delta$.

Let $V$ be a finite dimensional representation of the group $\GQpD$ over $\mathbb{F}_p$. The basechange $E^{sep}_\Delta\otimes_{\mathbb{F}_p}V$ is equipped with the diagonal semilinear action of $\GQpD$ and with the Frobenii $\varphi_\alpha$ for $\alpha\in \Delta$. These all commute with each other. We define the value of the functor $\mathbb{D}$ at $V$ by putting
\begin{equation*}
\mathbb{D}(V):=(E^{sep}_\Delta\otimes_{\mathbb{F}_p}V)^{\HQpD}\ .
\end{equation*}
By Proposition \ref{invEsepD} $\mathbb{D}(V)$ is a module over $E_\Delta$ inheriting the action of the monoid $T_{+,\Delta}$ from the action of $\varphi_\alpha$ ($\alpha\in\Delta$) and the Galois group $\GQpD$ on $E^{sep}_\Delta\otimes_{\mathbb{F}_p}V$. Our key Lemma is the following.

\begin{lem}\label{galinvbasis}
The $E^{sep}_\Delta$-module $E^{sep}_\Delta\otimes_{\mathbb{F}_p}V$ admits a basis consisting of elements fixed by $\HQpD$.
\end{lem}
\begin{proof}
At first consider the $E^{sep}_{\Delta,\circ}$-module $E^{sep}_{\Delta,\circ}\otimes_{\mathbb{F}_p}V$. We show by induction on $|\Delta|$ that $E^{sep}_{\Delta,\circ}\otimes_{\mathbb{F}_p}V$ admits a basis consisting of $\HQpD$-invariant vectors. The statement follows from this noting that $E^{sep}_{\Delta,\circ}$ is a subring in $E^{sep}_\Delta$ therefore the required basis exists also in $E^{sep}_\Delta\otimes_{\mathbb{F}_p}V\cong E^{sep}_\Delta\otimes_{E^{sep}_{\Delta,\circ}}(E^{sep}_{\Delta,\circ}\otimes_{\mathbb{F}_p}V)$.

By Hilbert's Thm.\ 90 the $\HQpa$-module $E^{sep}_\alpha\otimes_{\mathbb{F}_p}V$ is trivial for each $\alpha\in\Delta$. So we have an $E^{sep}_\alpha$-basis $e_1^{(\alpha)},\dots,e_d^{(\alpha)}$ of $E^{sep}_\alpha\otimes_{\mathbb{F}_p}V$ consisting of $\HQpa$-invariant elements. Since we have an action of the direct product $\HQpD$ on $V$, the $E_\alpha$-vector space
\begin{equation*}
V_\alpha:=E_\alpha e_1^{(\alpha)}+\dots + E_\alpha e_d^{(\alpha)}=(E^{sep}_\alpha\otimes_{\mathbb{F}_p}V)^{\HQpa}
\end{equation*}
admits a linear action of the group $\HQpDa$. Now note that the representations $V$ and $V_\alpha$ of the group $\HQpDa$ become isomorphic over the field $E^{sep}_\alpha$ by construction. Since $\HQpDa$ acts through a finite quotient on $V$, there is a finite extension $E'_\alpha$ of $E_\alpha$ contained in $E^{sep}_\alpha$ such that we have an isomorphism $E'_\alpha\otimes_{\mathbb{F}_p}V\cong E'_\alpha\otimes_{E_\alpha}V_\alpha$ of $\HQpDa$-representations. Making this identification and writing $e_i:=1\otimes e_i\in E'_\alpha\otimes_{\mathbb{F}_p}V$ (resp.\ $e_i^{(\alpha)}:=1\otimes e_i^{(\alpha)}$), $i=1,\dots,d$, for a basis $e_1,\dots,e_d$ in $V$ (resp.\ for the basis $e_1^{(\alpha)},\dots e_d^{(\alpha)}$ in $V_\alpha$) by an abuse of notation, we find a matrix $B\in \GL_d(E'_\alpha)$ with $B\rho(h)=\rho_\alpha(h)B$ for all $h\in \HQpDa$ where $\rho(h)\in \GL_d(\mathbb{F}_p)$ (resp.\ $\rho_\alpha(h)\in \GL_d(E_\alpha)$) is the matrix of the action of $h$ on $V$ (resp.\ on $V_\alpha$) in the basis $e_1,\dots,e_d$ (resp.\ $e_1^{(\alpha)},\dots e_d^{(\alpha)}$). Now $E'_\alpha/E_\alpha$ is a finite separable extension, so there exists a primitive element $u\in E'_\alpha$ with $E'_\alpha=E_\alpha(u)$. Hence we may write $B$ as a sum $B=B(u)=B_0+B_1u+\dots +B_{n-1}u^{n-1}$ for some matrices $B_0,B_1,\dots,B_{n-1}\in E_\alpha^{d\times d}$ with $n:=|E'_\alpha:E_\alpha|$. Since $\det B\neq 0$, the polynomial $\det(B(x)):=\det(B_0+B_1x+\dots +B_{n-1}x^{n-1})\in E_\alpha[x]$ is not identically $0$. As $E_\alpha$ is an infinite field, there exists a $u_0\in E_\alpha$ with $\det B(u_0)\neq 0$. Now we have $\rho(h)=B(u_0)^{-1}\rho_\alpha(h)B(u_0)$ for all $h\in \HQpDa$, ie.\ the representations $V$ and $V_\alpha$ of $\HQpDa$ are isomorphic already over $E_\alpha$. This shows that there exists a basis $v_1^{(\alpha)},\dots v_d^{(\alpha)}$ in $V_\alpha$ such that the action of each $h$ in $\HQpDa$ is given by a matrix in $\GL_d(\mathbb{F}_p)$ in this basis. We put 
\begin{align*}
V_{\alpha\ast}:=\mathbb{F}_pv_1^{(\alpha)}+\dots +\mathbb{F}_p v_d^{(\alpha)}\subset V_\alpha=\left(E^{sep}_\alpha\otimes_{\mathbb{F}_p} V\right)^{\HQpa}=\\
=\left(\left(\bigotimes_{\beta\in\Delta\setminus\{\alpha\}}1\right)\otimes (E^{sep}_\alpha\otimes_{\mathbb{F}_p} V)\right)^{\HQpa}\subseteq \left(E^{sep}_{\Delta,\circ}\otimes_{\mathbb{F}_p}V\right)^{\HQpa} \ .
\end{align*}
By induction we find a basis $v_1,\dots,v_n$ of $E^{sep}_{\Delta\setminus\{\alpha\},\circ}\otimes_{\mathbb{F}_p}V_{\alpha\ast}\subseteq \left(E^{sep}_{\Delta,\circ}\otimes_{\mathbb{F}_p}V\right)^{\HQpa}$ consisting of $\HQpDa$-invariant elements which are $\HQpa$-invariant, as well, by construction. Therefore $v_1,\dots,v_n$ is an $\HQpD$-invariant basis of $E^{sep}_{\Delta,\circ}\otimes_{\mathbb{F}_p}V$ as required.
\end{proof}

\begin{lem}\label{invertibleelts}
We have $(E^{sep}_\Delta)^\times\cap E_\Delta=E_\Delta^{\times}$.
\end{lem}
\begin{proof}
Let $u$ be arbitrary in $(E^{sep}_\Delta)^\times\cap E_\Delta$. Since $u$ is invariant under the action of $\HQpD$, so is its inverse $u^{-1}$ whence it also lies in $E_\Delta$ by Proposition \ref{invEsepD}.
\end{proof}

\begin{lem}\label{phiinv}
We have $\bigcap_{\alpha\in\Delta}(E^{sep}_\Delta)^{\varphi_\alpha=\id}=\mathbb{F}_p$.
\end{lem}
\begin{proof}
The containment $\mathbb{F}_p\subseteq \bigcap_{\alpha\in\Delta}(E^{sep}_\Delta)^{\varphi_\alpha=\id}\subseteq (E^{sep}_\Delta)^{\varphi_s=\id}$ is obvious. On the other hand, let $u\in E^{sep}_\Delta$ be arbitrary such that $\varphi_\alpha(u)=u$ for all $\alpha\in\Delta$. Then we also have $u^p=\varphi_s(u)=u$ as $\varphi_s$ is the absolute Frobenius on $E^{sep}_\Delta$. Since $E^{sep}_\Delta$ is defined as an inductive limit, $u$ lies in $E'_\Delta\cong (\bigotimes_{\alpha\in\Delta,\mathbb{F}_p}\mathbb{F}_{q_\alpha})\bs X'_\alpha\mid \alpha\in\Delta\js[X_\Delta^{-1}]$ for some collection $E'_\alpha=\mathbb{F}_{q_\alpha}\bg X'_\alpha\jg$ ($\alpha\in\Delta$) of finite separable extensions of $E_\alpha$. Note that $\bigotimes_{\alpha\in\Delta,\mathbb{F}_p}\mathbb{F}_{q_\alpha}$ is a finite \'etale algebra over $\mathbb{F}_p$, in particular, it is reduced. Therefore we have $|u^p|_{prod}=|u|_{prod}^p$. We deduce $|u|_{prod}=1$ unless $u=0$. In particular, $u$ lies in $E'^+_\Delta=(\bigotimes_{\alpha\in\Delta,\mathbb{F}_p}\mathbb{F}_{q_\alpha})\bs X'_\alpha\mid \alpha\in\Delta\js$. The constant term $u_0\in \bigotimes_{\alpha\in\Delta,\mathbb{F}_p}\mathbb{F}_{q_\alpha}$ also satisfies $\varphi_\alpha(u_0)=u_0$ for all $\alpha\in\Delta$. For a fixed $\alpha\in \Delta$ we choose an $\mathbb{F}_p$-basis $d_1,\dots,d_n$ of $\bigotimes_{\beta\in\Delta\setminus\{\alpha\},\mathbb{F}_p}\mathbb{F}_{q_\beta}$ and write $u_0=\sum_{i=1}^nc_i\otimes d_i$ with $c_i\in\mathbb{F}_{q_\alpha}$. This decomposition is unique and we compute 
\begin{align*}
\sum_{i=1}^nc_i\otimes d_i=u_0=\varphi_\alpha(u_0)=\sum_{i=1}^nc_i^p\otimes d_i\ . 
\end{align*}
We deduce $c_i=c_i^p$, ie.\ $c_i\in\mathbb{F}_p$ for all $1\leq i\leq n$. It follows by induction on $|\Delta|$ that $u_0$ lies in $\mathbb{F}_p$. Now $u-u_0$ is also fixed by each $\varphi_\alpha$ ($\alpha\in\Delta$), but we have $|u-u_0|_{prod}<1$. This implies by the discussion above that $u=u_0$ is in $\mathbb{F}_p$ as desired.
\end{proof}

\begin{pro}\label{mathbbD}
$\mathbb{D}(V)$ is an \'etale $T_{+,\Delta}$-module over $E_\Delta$ of rank $d:=\dim_{\mathbb{F}_p}V$. Moreover, we have $E^{sep}_\Delta\otimes_{E_\Delta}\mathbb{D}(V)\cong E^{sep}_\Delta\otimes_{\mathbb{F}_p}V$ and $$V=\bigcap_{\alpha\in\Delta}(E^{sep}_\Delta\otimes_{E_\Delta}\mathbb{D}(V))^{\varphi_\alpha=\id}\ .$$
\end{pro}
\begin{proof}
By Lemmata \ref{invEsepD} and \ref{galinvbasis} $\mathbb{D}(V)$ is a free module of rank $d$ over $E_\Delta$. Moreover, the matrix of $\varphi_\alpha$ in any basis of $\mathbb{D}(V)$ is invertible in $E^{sep}_\Delta$, therefore also in $E_\Delta$ by Lemma \ref{invertibleelts}. So the action of $T_{+,\Delta}$ on $\mathbb{D}(V)$ is \'etale. The last statement is a direct consequence of Lemmata \ref{galinvbasis} and \ref{phiinv}.
\end{proof}

\begin{lem}\label{tannakaDmodp}
For objects $V,V_1,V_2$ in $\RepDFp$ we have $\mathbb{D}(V_1\otimes_{\mathbb{F}_p}V_2)\cong \mathbb{D}(V_1)\otimes_{E_\Delta}\mathbb{D}(V_2)$ and $\mathbb{D}(V^*)\cong \mathbb{D}(V)^*$.
\end{lem}
\begin{proof}
We compute
\begin{align*}
\mathbb{D}(V_1\otimes_{\Fp}V_2)=\left(E^{sep}_\Delta\otimes_{\Fp} V_1\otimes_{\Fp}V_2\right)^{\HQpD}\cong \left((E^{sep}_\Delta\otimes_{\Fp} V_1)\otimes_{E^{sep}_\Delta}(E^{sep}_\Delta\otimes_{\Fp}V_2)\right)^{\HQpD}\cong\\
\left((E^{sep}_\Delta\otimes_{E_\Delta} \mathbb{D}(V_1))\otimes_{E^{sep}_\Delta}(E^{sep}_\Delta\otimes_{E_\Delta}\mathbb{D}(V_2))\right)^{\HQpD}\cong\\ \cong\left(E^{sep}_\Delta\otimes_{E_\Delta} (\mathbb{D}(V_1)\otimes_{E_\Delta}\mathbb{D}(V_2))\right)^{\HQpD}\cong \mathbb{D}(V_1)\otimes_{E_\Delta}\mathbb{D}(V_2)\ .
\end{align*}
For the second statement we have
\begin{align*}
\mathbb{D}(V^*)=\left(E^{sep}_\Delta\otimes_{\Fp}\Hom_{\Fp}(V,\Fp)\right)^{\HQpD}\cong \Hom_{E^{sep}_\Delta}(E^{sep}_\Delta\otimes_{\Fp}V,E^{sep}_\Delta)^{\HQpD}\cong\\ \cong \Hom_{E^{sep}_\Delta}(E^{sep}_\Delta\otimes_{E_\Delta}\mathbb{D}(V),E^{sep}_\Delta)^{\HQpD}\cong \left(E^{sep}_\Delta\otimes_{E_\Delta}\Hom_{E_\Delta}(\mathbb{D}(V),E_\Delta)\right)^{\HQpD}\cong \mathbb{D}(V)^*\ .
\end{align*}
\end{proof}

\begin{thm}\label{fullfaithmodp}
$\mathbb{D}$ is a fully faithful tensor functor from the category $\RepDFp$ to the category $\mathcal{D}^{et}(\varphi_{\Delta},\Gamma_\Delta,E_\Delta)$.
\end{thm}
\begin{proof}
Let $f\colon V_1\to V_2$ be a nonzero morphism in $\RepDFp$. Then the $E^{sep}_\Delta$-linear map $\id\otimes f\colon E^{sep}_\Delta\otimes_{\mathbb{F}_p} V_1\to E^{sep}_\Delta\otimes_{\mathbb{F}_p}V_2$ is also nonzero. By the last statement in Prop.\ \ref{mathbbD} it follows that $\mathbb{D}(f)\neq 0$ therefore the faithfulness.

Now let $V_1$ and $V_2$ be arbitrary objects in $\RepDFp$ and $\theta\colon \mathbb{D}(V_1)\to \mathbb{D}(V_2)$ be a morphism in $\mathcal{D}^{et}(\varphi_{\Delta},\Gamma_\Delta,E_\Delta)$. Then by Prop.\ \ref{mathbbD} we obtain a $\GQpD$-equivariant $\mathbb{F}_p$-linear map 
\begin{equation*}
f\colon V_1=\bigcap_{\alpha\in\Delta}\left(E^{sep}_\Delta\otimes_{E_\Delta} \mathbb{D}(V_1)\right)^{\varphi_\alpha=\id}\to \bigcap_{\alpha\in\Delta}\left(E^{sep}_\Delta\otimes_{E_\Delta} \mathbb{D}(V_2)\right)^{\varphi_\alpha=\id}=V_2
\end{equation*}
induced by $\theta$ for which we have $\theta=\mathbb{D}(f)$. Therefore $\mathbb{D}$ is full. The compatibility with tensor products is proven in Lemma \ref{tannakaDmodp}.
\end{proof}

\begin{rem}
Note that any \'etale $T_{+,\Delta}$-module $D$ in the image of the functor $\mathbb{D}$ is free as a module over $E_\Delta$ by construction.
\end{rem}

Consider the diagonal embedding $\mathrm{diag}\colon G_{\mathbb{Q}_p}\hookrightarrow G_{\mathbb{Q}_p,\Delta}$ sending $g\in G_{\mathbb{Q}_p}$ to $(g,\dots,g)$. This defines a functor $\widehat{\mathrm{diag}}\colon \RepDFp\to \operatorname{Rep}_{\mathbb{F}_p}(G_{\mathbb{Q}_p})$ via restriction. On the other hand, we have the reduction map $\ell\colon \mathcal{D}^{et}(\varphi_{\Delta},\Gamma_\Delta,E_\Delta)\to \mathcal{D}^{et}(\varphi,\Gamma,E)$ to usual $(\varphi,\Gamma)$-modules defined in section 2.4 of \cite{MultVar}. Recall that this is given by taking the quotient by the ideal generated by $(X_\alpha-X_\beta\mid \alpha,\beta\in\Delta)$ and restricting to the diagonal $\varphi=\varphi_s=\prod_{\alpha\in\Delta}\varphi_\alpha$ and $\Gamma:=\{(\gamma,\dots,\gamma)\}\leq \Gamma_\Delta$.

\begin{cor}
There is a natural isomorphism $\widehat{\mathrm{diag}}\cong \mathbb{V}_F\circ\ell\circ\mathbb{D}$ of functors $\RepDFp\to \operatorname{Rep}_{\mathbb{F}_p}(G_{\mathbb{Q}_p})$ where $\mathbb{V}_F\colon \mathcal{D}^{et}(\varphi,\Gamma,E)\to \operatorname{Rep}_{\mathbb{F}_p}(G_{\mathbb{Q}_p})$ is Fontaine's functor from classical \'etale $(\varphi,\Gamma)$-modules to Galois representations. 
\end{cor}
\begin{proof}
We may identify $E_\alpha\overset{\sim}{\to} E=\mathbb{F}_p\bg X\jg$ by sending $X_\alpha\to X$ for all $\alpha\in\Delta$. We extend this identification to $E^{sep}_\alpha\to E^{sep}$. So we obtain a map $\ell^{sep}\colon E^{sep}_\Delta\to E^{sep}$ sending each subring $E^{sep}_\alpha$ to $E^{sep}$ via these identifications and completing on the level of each finite extension $E'_\Delta$. Then $\ell^{sep}$ is $G_{\Qp}$-equivariant where $G_{\Qp}$ acts on $E^{sep}_\Delta$ via the diagonal embedding $G_{\mathbb{Q}_p}\hookrightarrow G_{\mathbb{Q}_p,\Delta}$ and the usual way on $E^{sep}$. The restriction of $\ell^{sep}$ to $E_\Delta$ is the map $\ell\colon E_\Delta\to E$ defined above, so the diagram
\begin{equation*}
\xymatrix{
E_\Delta\ar@{^{(}->}[r]\ar[d]_{\ell} & E_\Delta^{sep}\ar[d]^{\ell^{sep}}\\
E\ar@{^{(}->}[r] & E^{sep}
}
\end{equation*}
commutes. Thus for an object $V$ in $\RepDFp$ we compute
\begin{align*}
\mathbb{V}_F\circ\ell\circ\mathbb{D}(V)=\mathbb{V}_F(E\otimes_{\ell,E_\Delta}\mathbb{D}(V))=\mathbb{V}_F((E^{sep})^{H_{\mathbb{Q}_p}}\otimes_{\ell,E_\Delta}\mathbb{D}(V))=\\
=\mathbb{V}_F((E^{sep}\otimes_{\ell^{sep},E^{sep}_\Delta}E^{sep}_\Delta\otimes_{E_\Delta}\mathbb{D}(V))^{H_{\mathbb{Q}_p}})=\mathbb{V}_F((E^{sep}\otimes_{\ell^{sep},E^{sep}_\Delta}E^{sep}_\Delta\otimes_{\mathbb{F}_p}V)^{H_{\mathbb{Q}_p}})=\\
=\mathbb{V}_F((E^{sep}\otimes_{\mathbb{F}_p}V)^{H_{\mathbb{Q}_p}})=\mathbb{V}_F\circ\mathbb{D}_F(V)=V\mid_{\mathrm{diag}(G_{\mathbb{Q}_p})}=\widehat{\diag}(V)\ ,
\end{align*}
where $\mathbb{D}_F\colon \operatorname{Rep}_{\mathbb{F}_p}(G_{\mathbb{Q}_p})\to \mathcal{D}^{et}(\varphi,\Gamma,E)$ stands for Fontaine's classical functor. 
\end{proof}

\subsection{The functor $\mathbb{V}$}

In order to show that the functor $\mathbb{D}$ is essentially surjective, we construct its quasi-inverse $\mathbb{V}$. Let $D$ be an object in $\mathcal{D}^{et}(\varphi_{\Delta},\Gamma_\Delta,E_\Delta)$. The group $\GQpD$ acts on $E_\Delta^{sep}\otimes_{E_\Delta}D$ via the formula $g(\lambda\otimes x):=g(\lambda)\otimes \chi_{cyc}(g)(x)$ ($g\in\GQpD$, $\lambda\in E_\Delta^{sep}$, $x\in D$) where $\chi_{cyc}\colon \GQpD\to \Gamma_\Delta$ is the quotient map. Moreover, each partial Frobenius $\varphi_\alpha$ ($\alpha\in \Delta$) acts semilinearly on $E_\Delta^{sep}\otimes_{E_\Delta}D$ via the formula $\varphi_\alpha(\lambda\otimes x):=\varphi_\alpha(\lambda)\otimes \varphi_\alpha(x)$. All these actions commute with each other by construction. We define
\begin{equation*}
\mathbb{V}(D):=\bigcap_{\alpha\in\Delta}\left(E_\Delta^{sep}\otimes_{E_\Delta}D\right)^{\varphi_\alpha=\id}\ .
\end{equation*}
$\mathbb{V}(D)$ is a---a priori not necessarily finite dimensional---representation of $\GQpD$ over $\mathbb{F}_p$. 

\begin{lem}\label{polinv}
For any integer $r>0$ we have $\bigcap_{\beta\in\Delta}(E^{sep}_{\Delta\setminus\{\alpha\}}[X_\alpha]/(X_\alpha^r))^{\varphi_\beta=\id}=\mathbb{F}_p[X_\alpha]/(X_\alpha^r)$.
\end{lem}
\begin{proof}
This follows from Lemma \ref{phiinv} noting that $\mathbb{F}_p[X_\alpha]/(X_\alpha^r)$ is a finite dimensional $\mathbb{F}_p$-vector space on which $\varphi_\beta$ acts identically for all $\beta\in\Delta\setminus\{\alpha\}$ and we have $E^{sep}_{\Delta\setminus\{\alpha\}}[X_\alpha]/(X_\alpha^r)\cong E^{sep}_{\Delta\setminus\{\alpha\}}\otimes_{\mathbb{F}_p}\mathbb{F}_p[X_\alpha]/(X_\alpha^r)$.
\end{proof}

\begin{lem}\label{tensorpol}
For any integer $r>0$ and finitely generated $E^+_{\overline{\alpha}}/(X_\alpha^r)$-module $M$ we have an identification $E^{sep}_{\Delta\setminus\{\alpha\}}[X_\alpha]/(X_\alpha^r)\otimes_{E^+_{\overline{\alpha}}/(X_\alpha^r)}M\cong E^{sep}_{\Delta\setminus\{\alpha\}}\otimes_{E_{\Delta\setminus\{\alpha\}}}M$.
\end{lem}
\begin{proof}
This follows from the isomorphism $E^+_{\overline{\alpha}}/(X_\alpha^r)\cong E_{\Delta\setminus\{\alpha\}}[X_\alpha]/(X_\alpha^r)$.
\end{proof}

For a subset $S\subseteq\Delta$ we put $E^{sep+}_S:=\varinjlim E'^+_S$ so we have $E^{sep}_S=E^{sep+}_S[X_S^{-1}]$.

\begin{lem}\label{Esepflat}
$E^{sep}_S$ (resp.\ $E^{sep+}_S$) is flat as a module over $E_S$ (resp.\ over $E^+_S$) for all $S\subseteq \Delta$.
\end{lem}
\begin{proof}
By construction, $E'_S$ (resp.\ $E'^+_S$) is finite free over $E_S$ (resp.\ over $E^+_S$), so $E^{sep}_S$ (resp.\ $E^{sep+}_S$) is the direct limit of flat modules hence flat.
\end{proof}

\begin{lem}\label{complinv}
We have $(E^{sep+}_{\Delta\setminus\{\alpha\}}\bs X_\alpha\js [X_\Delta^{-1}])^{\HQpDa}=E_\Delta$.
\end{lem}
\begin{proof}
We have $E_\Delta=E^+_{\Delta\setminus\{\alpha\}}\bs X_\alpha\js [X_\Delta^{-1}]$ where $E^+_{\Delta\setminus\{\alpha\}}=(E^{sep+}_{\Delta\setminus\{\alpha\}})^{\HQpDa}$ by Lemma \ref{invEsepD} and $\HQpDa$ acts trivially on both $X_\alpha$ and $X_\Delta$, so acts on the power series ring $E^{sep+}_{\Delta\setminus\{\alpha\}}\bs X_\alpha\js$ coefficientwise.
\end{proof}

Our main result in this section is the following

\begin{thm}\label{modpequiv}
The functors $\mathbb{D}$ and $\mathbb{V}$ are quasi-inverse equivalences of categories between the Tannakian categories $\RepDFp$ and $\mathcal{D}^{et}(\varphi_{\Delta},\Gamma_\Delta,E_\Delta)$.
\end{thm}

\begin{cor}\label{alwaysfree}
Any object $D$ in $\mathcal{D}^{et}(\varphi_{\Delta},\Gamma_\Delta,E_\Delta)$ is a free module over $E_\Delta$.
\end{cor}
\begin{proof}
This follows from the essential surjectivity of $\mathbb{D}$ using the remark after Thm.\ \ref{fullfaithmodp}.
\end{proof}

\begin{proof}[Proof of Thm.\ \ref{modpequiv}]
This is a long proof that we divide into $5$ steps.

\emph{Step 1. Reducing the statement to the essential surjectivity of $\mathbb{D}$.} By Thm.\ \ref{fullfaithmodp} the functor $\mathbb{D}$ is fully faithful and we have $\mathbb{V}\circ\mathbb{D}(V)\cong V$ naturally in $V$ for any object $V$ in $\RepDFp$ by Prop.\ \ref{mathbbD}. Moreover, by Lemma \ref{tannakaDmodp} $\mathbb{D}$ is compatible with tensor products and duals. So it remains to show that $\mathbb{D}$ is essentially surjective. We proceed by induction on $|\Delta|$. For $|\Delta|=1$ this is a classical result of Fontaine (see e.g.\ Thm.\ 2.21 in \cite{FO}). Suppose that $|\Delta|>1$, fix $\alpha\in\Delta$, and pick an object $D$ in $\mathcal{D}^{et}(\varphi_{\Delta},\Gamma_\Delta,E_\Delta)$.

\emph{Step 2. The goal here is to trivialize the $\varphi_\beta$-action ($\beta\in\Delta\setminus\{\alpha\}$) on $D^{+*}_{\overline{\alpha}}/X_\alpha^rD^{+*}_{\overline{\alpha}}$ uniformly in $r$ by tensoring up with $E^{sep}_{\Delta\setminus\{\alpha\}}$.}
By Prop.\ \ref{alphaintetale} $D^{+*}_{\overline{\alpha}}$ is an \'etale $T_{+,\overline{\alpha}}$-module over $E^+_{\overline{\alpha}}$. Reducing mod $X_\alpha^r$ for an integer $r>0$ we deduce that $D^{+*}_{\overline{\alpha},r}:=D^{+*}_{\overline{\alpha}}/X_\alpha^rD^{+*}_{\overline{\alpha}}$ is an \'etale $T_{+,\overline{\alpha}}$-module over $E^+_{\overline{\alpha}}/(X_\alpha^r)\cong E_{\Delta\setminus\{\alpha\}}[X_\alpha]/(X_\alpha^r)$. Since each $\varphi_\beta$ ($\beta\in\Delta\setminus\{\alpha\}$) acts trivially on the variable $X_\alpha$, we have a natural isomorphism of functors
\begin{equation*}
E_{\Delta\setminus\{\alpha\}}[X_\alpha]/(X_\alpha^r)\otimes_{E_{\Delta\setminus\{\alpha\}}[X_\alpha]/(X_\alpha^r),\varphi_t}\cdot\cong E_{\Delta\setminus\{\alpha\}}\otimes_{E_{\Delta\setminus\{\alpha\}},\varphi_t}\cdot
\end{equation*}
for all $t\in T_{+,\overline{\alpha}}$. Hence $D^{+*}_{\overline{\alpha},r}$ is an object in $\mathcal{D}^{et}(\varphi_{{\Delta\setminus\{\alpha\}}},\Gamma_{\Delta\setminus\{\alpha\}},E_{\Delta\setminus\{\alpha\}})$ since $E_{\Delta\setminus\{\alpha\}}[X_\alpha]/(X_\alpha^r)$ is finitely generated as a module over $E_{\Delta\setminus\{\alpha\}}$. By the inductional hypothesis (see step $1$), we can therefore trivialize $D^{+*}_{\overline{\alpha},r}$ by tensoring with $E^{sep}_{\Delta\setminus\{\alpha\}}$ over $E_{\Delta\setminus\{\alpha\}}$. However, this is the same as applying $E^{sep}_{\Delta\setminus\{\alpha\}}[X_\alpha]/(X_\alpha^r)\otimes_{E_{\Delta\setminus\{\alpha\}}[X_\alpha]/(X_\alpha^r)}\cdot$ by Lemma \ref{tensorpol}. Hence the natural map
\begin{align}
E^{sep}_{\Delta\setminus\{\alpha\}}[X_\alpha]/(X_\alpha^r)\otimes_{\mathbb{F}_p[X_\alpha]/(X_\alpha^r)}\bigcap_{\beta\in\Delta\setminus\{\alpha\}}\left(E^{sep}_{\Delta\setminus\{\alpha\}}[X_\alpha]/(X_\alpha^r)\otimes_{E^+_{\overline{\alpha}}/(X_\alpha^r)}D^{+*}_{\overline{\alpha},r}\right)^{\varphi_\beta=\id}\overset{\sim}{\to}\notag\\
\overset{\sim}{\to} E^{sep}_{\Delta\setminus\{\alpha\}}[X_\alpha]/(X_\alpha^r)\otimes_{E^+_{\overline{\alpha}}/(X_\alpha^r)}D^{+*}_{\overline{\alpha},r}\cong E^{sep}_{\Delta\setminus\{\alpha\}}[X_\alpha]/(X_\alpha^r)\otimes_{E^+_{\overline{\alpha}}}D^{+*}_{\overline{\alpha}}\label{modrisom}
\end{align}
is an isomorphism for all $r>0$ using Lemma \ref{polinv}. Our key Lemma is the following consequence of Prop.\ \ref{alphaintetale}.
\begin{lem}\label{bounded}
There exists a finitely generated $E_\Delta^+$-submodule $M\leq D^{+*}_{\overline{\alpha}}$ such that
\begin{equation}
\bigcap_{\beta\in\Delta\setminus\{\alpha\}}\left(E^{sep}_{\Delta\setminus\{\alpha\}}[X_\alpha]/(X_\alpha^r)\otimes_{E^+_{\overline{\alpha}}}D^{+*}_{\overline{\alpha},}\right)^{\varphi_\beta=\id}\label{rfix}
\end{equation}
is contained in the image of the map 
\begin{equation}\label{maplevelr}
E^{sep+}_{\Delta\setminus\{\alpha\}}[X_\alpha]/(X_\alpha^r)\otimes_{E^+_{\Delta}}M\to E^{sep+}_{\Delta\setminus\{\alpha\}}[X_\alpha]/(X_\alpha^r)\otimes_{E^+_{\Delta}}D^{+*}_{\overline{\alpha}}\cong E^{sep}_{\Delta\setminus\{\alpha\}}[X_\alpha]/(X_\alpha^r)\otimes_{E^+_{\overline{\alpha}}}D^{+*}_{\overline{\alpha}}
\end{equation}
induced by the inclusion $M\leq D^{+*}_{\overline{\alpha}}$ for all $r>0$. Moreover, $M$ can be chosen in such a way that \eqref{maplevelr} is injective.
\end{lem}
\begin{proof}
We show that $M:=X_{\Delta\setminus\{\alpha\}}^{-1}D_0$ will do where $D_0$ is defined in Lemma \ref{phicover}. Since $D_0$ is finitely generated over $E^+_\Delta$, so is $M$. By Lemma \ref{phicover}, we have $D^{+*}_{\overline{\alpha}}=\bigcup_{l\geq 0} E^+_\Delta\varphi_{\overline{\alpha}}^l(M)$. For any fixed $r>0$ there exists an integer $l_r\geq 0$ such that \eqref{rfix} is contained in 
\begin{align*}
E^{sep+}_{\Delta\setminus\{\alpha\}}[X_\alpha]/(X_\alpha^r)\otimes_{E^+_{\Delta}}X_{\Delta\setminus\{\alpha\}}^{-p^{l_r}+1}M\subseteq E^{sep+}_{\Delta\setminus\{\alpha\}}[X_\alpha]/(X_\alpha^r)\otimes_{E^+_{\Delta}}E^+_{\Delta}\varphi_{\overline{\alpha}}^{l_r}(M)=\\
=E^{sep+}_{\Delta\setminus\{\alpha\}}[X_\alpha]/(X_\alpha^r)\varphi_{\overline{\alpha}}^{l_r}(E^{sep+}_{\Delta\setminus\{\alpha\}}[X_\alpha]/(X_\alpha^r)\otimes_{E^+_{\Delta}}M)\ . 
\end{align*}
Now if $x$ lies in \eqref{rfix}, then we have $\varphi_{\overline{\alpha}}^{l_r}(x)=x$. On the other hand, $x$ lies in $$E'_{\Delta\setminus\{\alpha\}}[X_\alpha]/(X_\alpha^r)\varphi_{\overline{\alpha}}^{l_r}(E'_{\Delta\setminus\{\alpha\}}[X_\alpha]/(X_\alpha^r)\otimes_{E^+_{\Delta}}M)$$ for some finite separable extensions $E'_\beta/E_\beta$ for $\beta\in\Delta\setminus\{\alpha\}$ and $E'_{\Delta\setminus\{\alpha\}}:=\widehat{\bigotimes}_{\beta\in\Delta\setminus\{\alpha\},\mathbb{F}_p}E'_\beta$. Therefore $x$ lies in fact in $E'_{\Delta\setminus\{\alpha\}}[X_\alpha]/(X_\alpha^r)\otimes_{E^+_{\Delta}}M$ by the injectivity of the map
\begin{align*}
\id\otimes\varphi_{\overline{\alpha}}^{l_r}\colon E'_{\Delta\setminus\{\alpha\}}[X_\alpha]/(X_\alpha^r)\otimes_{E'_{\Delta\setminus\{\alpha\}}[X_\alpha]/(X_\alpha^r),\varphi_{\overline{\alpha}}^{l_r}}(E'_{\Delta\setminus\{\alpha\}}[X_\alpha]/(X_\alpha^r)\otimes_{E^+_{\overline{\alpha}}}D^{+*}_{\overline{\alpha}})\to\\
\to E'_{\Delta\setminus\{\alpha\}}[X_\alpha]/(X_\alpha^r)\otimes_{E^+_{\overline{\alpha}}}D^{+*}_{\overline{\alpha}}
\end{align*}
($D^{+*}_{\overline{\alpha}}$ is \'etale) noting that the absolute Frobenius $\varphi_{\overline{\alpha}}\colon E'_{\Delta\setminus\{\alpha\}}\to  E'_{\Delta\setminus\{\alpha\}}$ is injective since the ring $ E'_{\Delta\setminus\{\alpha\}}$ is the localization of a power series ring over a finite \'etale algebra over $\mathbb{F}_p$, in particular, it is reduced. 

Finally, by the proof of Lemma \ref{phicover} we may choose $D_0=X_{\Delta\setminus\{\alpha\}}^{-k}(D^+\cap D^{+*}_{\overline{\alpha}})$ for some integer $k>0$ whence $M=X_{\Delta\setminus\{\alpha\}}^{-k-1}(D^+\cap D^{+*}_{\overline{\alpha}})$. So by Lemma \ref{twoprodinD+} $D^{+*}_{\overline{\alpha}}/M$ has no $X_\alpha$-torsion as $D^{+*}_{\overline{\alpha}}/M\cong D^{+*}_{\overline{\alpha}}+X_{\Delta\setminus\{\alpha\}}^{-k-1}D^+/(X_{\Delta\setminus\{\alpha\}}^{-k-1}D^+)$ is contained in $D^+_{\overline{\alpha}}/(X_{\Delta\setminus\{\alpha\}}^{-k-1}D^+)\cong D^+_{\overline{\alpha}}/D^+$. Therefore the map \eqref{maplevelr} is injective.
\end{proof}

\emph{Step 3. The goal here is to show the following compatibility of our construction with projective limits with respect to $r$.}
\begin{lem}\label{invlimcomm}
We have 
\begin{align*}
\varprojlim_r \left(E^{sep+}_{\Delta\setminus\{\alpha\}}[X_\alpha]/(X_\alpha^r)\otimes_{E^+_{\Delta}}M\right)\cong E^{sep+}_{\Delta\setminus\{\alpha\}}\bs X_\alpha\js\otimes_{E^+_{\Delta}}M\ ,\\
\varprojlim_r \left(E^{sep}_{\Delta\setminus\{\alpha\}}[X_\alpha]/(X_\alpha^r)\otimes_{E^+_{\overline{\alpha}}}D^{+*}_{\overline{\alpha}}\right)\cong E^{sep}_{\Delta\setminus\{\alpha\}}\bs X_\alpha\js\otimes_{E^+_{\overline{\alpha}}}D^{+*}_{\overline{\alpha}}\quad\text{, and}\\
\varprojlim_r \left(E^{sep}_{\Delta\setminus\{\alpha\}}[X_\alpha]/(X_\alpha^r)\otimes_{\mathbb{F}_p[X_\alpha]/(X_\alpha^r)}\bigcap_{\beta\in\Delta\setminus\{\alpha\}}\left(E^{sep}_{\Delta\setminus\{\alpha\}}[X_\alpha]/(X_\alpha^r)\otimes_{E^+_{\overline{\alpha}}/(X_\alpha^r)}D^{+*}_{\overline{\alpha},r}\right)^{\varphi_\beta=\id}\right)\cong \\
\cong  E^{sep}_{\Delta\setminus\{\alpha\}}\bs X_\alpha\js\otimes_{\mathbb{F}_p\bs X_\alpha\js}\bigcap_{\beta\in\Delta\setminus\{\alpha\}}\left(E^{sep}_{\Delta\setminus\{\alpha\}}\bs X_\alpha\js\otimes_{E_{\Delta}}D\right)^{\varphi_\beta=\id}\ .
\end{align*}
\end{lem}
\begin{proof}
Since $M$ is contained in $D$, $M$ has no $X_\alpha$-torsion. In particular, $M$ is flat as a module over the local ring $\mathbb{F}_p\bs X_\alpha\js$ and $\Tor_i^{\Fp\bs X_\alpha\js}(\mathbb{F}_p[X_\alpha]/(X_\alpha^r),M)=0$ for integers $i,r> 0$. Now we have the identification $$E^{sep+}_{\Delta\setminus\{\alpha\}}[X_\alpha]/(X_\alpha^r)\otimes_{E^+_{\Delta}}\cdot\cong E^{sep+}_{\Delta\setminus\{\alpha\}}\otimes_{E^+_{\Delta\setminus\{\alpha\}}}(\mathbb{F}_p[X_\alpha]/(X_\alpha^r)\otimes_{\mathbb{F}_p\bs X_\alpha\js}\cdot)$$ applied to an arbitrary projective resolution $P_\bullet$ of $M$ as an $E_\Delta^+$-module. Noting that each $P_j$ ($j\geq 0$) is flat over $\Fp\bs X_\alpha\js$ (as they are torsion-free) we deduce that $\mathbb{F}_p[X_\alpha]/(X_\alpha^r)\otimes_{\mathbb{F}_p\bs X_\alpha\js}P_\bullet$ is acyclic in nonzero degrees as it computes $\Tor_\bullet^{\Fp\bs X_\alpha\js}(\mathbb{F}_p[X_\alpha]/(X_\alpha^r),M)$. Moreover, by Lemma \ref{Esepflat} $E^{sep+}_{\Delta\setminus\{\alpha\}}$ is flat over $E^+_{\Delta\setminus\{\alpha\}}$ whence the complex $E^{sep+}_{\Delta\setminus\{\alpha\}}[X_\alpha]/(X_\alpha^r)\otimes_{E^+_{\Delta}}P_\bullet$ is also acyclic in nonzero degrees showing that $E^{sep+}_{\Delta\setminus\{\alpha\}}[X_\alpha]/(X_\alpha^r)$ and $M$ are $\Tor$-independent over $E^+_\Delta$.

On the other hand, $M$ is finitely generated over $E^+_\Delta$, so we have short exact sequences
\begin{align*}
0\to M_1\to (E^+_\Delta)^{k_0}\overset{f_0}{\to} M\to 0\qquad\text{and}\qquad
0\to M_2\to (E^+_\Delta)^{k_1}\to M_1\to 0
\end{align*}
by noetherianity. In order to simplify notation write $(\cdot)_r$ for $E^{sep+}_{\Delta\setminus\{\alpha\}}[X_\alpha]/(X_\alpha^r)\otimes_{E^+_{\Delta}}\cdot$ to obtain an exact sequence
\begin{align*}
(M_2)_r\to (E^+_\Delta)_r^{k_1}\overset{f_{1,r}}{\to} (E^+_\Delta)_r^{k_0}\overset{f_{0,r}}{\to} (M)_r\to 0
\end{align*}
for all $r>0$ using the $\Tor$-independence above. Now since the natural map $(N)_{r_1}\to (N)_{r_2}$ is surjective for any $E^+_\Delta$-module $N$ and $r_1\geq r_2>0$ by the right exactness of $\cdot\otimes_{E_\Delta^+}N$, the natural map $\Ker(f_{0,r_1})\to \Ker(f_{0,r_2})$ is also surjective (applying this in case $N=M_1$ and a diagram chasing). So the Mittag-Leffler property is satisfied for these projective systems showing that the map $\varprojlim_r f_{0,r}$ is surjective with kernel $\varprojlim_r \Ker(f_{0,r})=\varprojlim_r\mathrm{Im}(f_{1,r})$. Applying the same trick as above with $N=M_2$ we deduce that the projective system $\Ker(f_{1,r})$ also satisfies the Mittag-Leffler property showing that $\varprojlim_r f_{1,r}$ has image $\varprojlim_r \mathrm{Im}(f_{1,r})$. In particular, $\varprojlim_r (M)_r$ is the cokernel of the map $\varprojlim_r f_{1,r}\colon (E^{sep+}_{\Delta\setminus\{\alpha\}}\bs X_\alpha\js)^{k_1}\to (E^{sep+}_{\Delta\setminus\{\alpha\}}\bs X_\alpha\js)^{k_0}$ and so is $E^{sep+}_{\Delta\setminus\{\alpha\}}\bs X_\alpha\js\otimes_{E^+_\Delta}M$ as claimed. The second statement follows in exactly the same way.

For the third statement note that the isomorphism \eqref{modrisom} and the surjectivity of the map $E^{sep}_{\Delta\setminus\{\alpha\}}[X_\alpha]/(X_\alpha^{r_1})\otimes_{E^+_{\overline{\alpha}}}D^{+*}_{\overline{\alpha}}\to E^{sep}_{\Delta\setminus\{\alpha\}}[X_\alpha]/(X_\alpha^{r_2})\otimes_{E^+_{\overline{\alpha}}}D^{+*}_{\overline{\alpha}}$ implies that the map
\begin{align*}
\bigcap_{\beta\in\Delta\setminus\{\alpha\}}\left(E^{sep}_{\Delta\setminus\{\alpha\}}[X_\alpha]/(X_\alpha^{r_1})\otimes_{E^+_{\overline{\alpha}}/(X_\alpha^{r_1})}D^{+*}_{\overline{\alpha},r_1}\right)^{\varphi_\beta=\id}\to\\ \to\bigcap_{\beta\in\Delta\setminus\{\alpha\}}\left(E^{sep}_{\Delta\setminus\{\alpha\}}[X_\alpha]/(X_\alpha^{r_2})\otimes_{E^+_{\overline{\alpha}}/(X_\alpha^{r_2})}D^{+*}_{\overline{\alpha},r_2}\right)^{\varphi_\beta=\id}
\end{align*}
is also onto for all $r_1\geq r_2$. Therefore the natural map
\begin{align*}
\bigcap_{\beta\in\Delta\setminus\{\alpha\}}\left(E^{sep}_{\Delta\setminus\{\alpha\}}\bs X_\alpha\js\otimes_{E^+_{\overline{\alpha}}}D^{+*}_{\overline{\alpha}}\right)^{\varphi_\beta=\id}=\\
=\varprojlim_r\bigcap_{\beta\in\Delta\setminus\{\alpha\}}\left(E^{sep}_{\Delta\setminus\{\alpha\}}[X_\alpha]/(X_\alpha^{r})\otimes_{E^+_{\overline{\alpha}}/(X_\alpha^{r})}D^{+*}_{\overline{\alpha},r}\right)^{\varphi_\beta=\id}\to\\ \to\bigcap_{\beta\in\Delta\setminus\{\alpha\}}\left(E^{sep}_{\Delta\setminus\{\alpha\}}[X_\alpha]/(X_\alpha)\otimes_{E^+_{\overline{\alpha}}/(X_\alpha)}D^{+*}_{\overline{\alpha},1}\right)^{\varphi_\beta=\id}
\end{align*}
is also onto using the second statement of the Lemma. On the other hand, the kernel of this map equals 
\begin{align*}
\bigcap_{\beta\in\Delta\setminus\{\alpha\}}\left(E^{sep}_{\Delta\setminus\{\alpha\}}\bs X_\alpha\js\otimes_{E^+_{\overline{\alpha}}}D^{+*}_{\overline{\alpha}}\right)^{\varphi_\beta=\id}\cap X_\alpha E^{sep}_{\Delta\setminus\{\alpha\}}\bs X_\alpha\js\otimes_{E^+_{\overline{\alpha}}}D^{+*}_{\overline{\alpha}}=\\
=X_\alpha \bigcap_{\beta\in\Delta\setminus\{\alpha\}}\left(E^{sep}_{\Delta\setminus\{\alpha\}}\bs X_\alpha\js\otimes_{E^+_{\overline{\alpha}}}D^{+*}_{\overline{\alpha}}\right)^{\varphi_\beta=\id}
\end{align*}
since $X_\alpha$ is fixed by each $\varphi_\beta$ and $E^{sep}_{\Delta\setminus\{\alpha\}}\bs X_\alpha\js\otimes_{E^+_{\overline{\alpha}}}D^{+*}_{\overline{\alpha}}$ has no $X_\alpha$-torsion. This shows, in particular, that $\bigcap_{\beta\in\Delta\setminus\{\alpha\}}\left(E^{sep}_{\Delta\setminus\{\alpha\}}\bs X_\alpha\js\otimes_{E^+_{\overline{\alpha}}}D^{+*}_{\overline{\alpha}}\right)^{\varphi_\beta=\id}$ is finitely generated over $\mathbb{F}_p\bs X_\alpha\js$ by the topological Nakayama Lemma (see \cite{BH}). Moreover, it is torsion-free hence free as $E^{sep}_{\Delta\setminus\{\alpha\}}\bs X_\alpha\js\otimes_{E^+_{\overline{\alpha}}}D^{+*}_{\overline{\alpha}}$ has no $X_\alpha$-torsion either. In particular, $$E^{sep}_{\Delta\setminus\{\alpha\}}\bs X_\alpha\js\otimes_{\mathbb{F}_p\bs X_\alpha\js}\bigcap_{\beta\in\Delta\setminus\{\alpha\}}\left(E^{sep}_{\Delta\setminus\{\alpha\}}\bs X_\alpha\js\otimes_{E_{\Delta}}D\right)^{\varphi_\beta=\id}$$ is $X_\alpha$-adically complete and the result follows.
\end{proof}

\emph{Step 4. The goal here is to obtain a $(\varphi_\alpha,\Gamma_\alpha)$-module $D_\alpha$ over $E_\alpha$ (by trivializing the action of each $\varphi_\beta$, $\beta\in\Delta\setminus\{\alpha\}$) which is at the same time a linear representation of the group $\GQpDa$.}
We take projective limits of the inclusions in Lemma \ref{bounded} with respect to $r$ to conclude (using Lemma \ref{invlimcomm}) that
\begin{equation*}
\bigcap_{\beta\in\Delta\setminus\{\alpha\}}\left(E^{sep}_{\Delta\setminus\{\alpha\}}\bs X_\alpha\js\otimes_{E^+_{\overline{\alpha}}}D^{+*}_{\overline{\alpha}}\right)^{\varphi_\beta=\id}
\end{equation*}
is contained in the image of the map 
\begin{equation*}
E^{sep+}_{\Delta\setminus\{\alpha\}}\bs X_\alpha\js\otimes_{E^+_{\Delta}}M\to E^{sep}_{\Delta\setminus\{\alpha\}}\bs X_\alpha\js\otimes_{E^+_{\overline{\alpha}}}D^{+*}_{\overline{\alpha}}\ .
\end{equation*}
Note that $M[X_\Delta^{-1}]=D^{+*}_{\overline{\alpha}}[X_\Delta^{-1}]=D^{+*}_{\overline{\alpha}}[X_\alpha^{-1}]=D$ and $\varphi_\beta$ acts trivially on $X_\alpha$. So inverting $X_\Delta$ above we deduce that
\begin{equation*}
D_\alpha:=\bigcap_{\beta\in\Delta\setminus\{\alpha\}}\left(E^{sep}_{\Delta\setminus\{\alpha\}}\bg X_\alpha\jg\otimes_{E_{\Delta}}D\right)^{\varphi_\beta=\id}
\end{equation*}
is contained in the image of the map 
\begin{equation*}
E^{sep+}_{\Delta\setminus\{\alpha\}}\bs X_\alpha\js[X_\Delta^{-1}]\otimes_{E_{\Delta}}D\hookrightarrow E^{sep}_{\Delta\setminus\{\alpha\}}\bg X_\alpha\jg\otimes_{E_{\Delta}}D\ .
\end{equation*} 
On the other hand, by \eqref{modrisom} and the third statement of Lemma \ref{invlimcomm} we have an isomorphism
\begin{equation}
 E^{sep}_{\Delta\setminus\{\alpha\}}\bg X_\alpha\jg\otimes_{\mathbb{F}_p\bg X_\alpha\jg}D_\alpha\overset{\sim}{\to} E^{sep}_{\Delta\setminus\{\alpha\}}\bg X_\alpha\jg\otimes_{E_{\Delta}}D\ .\label{inftyisom}
\end{equation}
\begin{lem}
The finite dimensional $\mathbb{F}_p\bg X_\alpha\jg$-vector space $D_\alpha$ has the structure of an \'etale $(\varphi_\alpha,\Gamma_\alpha)$-module. At the same time it is a (linear) representation of the group $\GQpDa$. These two actions commute with each other.
\end{lem}
\begin{proof}
The operator $\varphi_\alpha$ and the groups $\Gamma_\alpha$ and $\GQpDa$ act naturally on $D_\alpha$. For the \'etaleness of the action of $\varphi_\alpha$ on $D_\alpha$ note that we have $\mathbb{F}_p\bg X_\alpha\jg\otimes_{\mathbb{F}_p\bg X_\alpha\jg,\varphi_\alpha}D\cong D$ by the \'etale property of $\varphi_\alpha$ on $D$ and that $\varphi_\beta$ acts trivially on $\mathbb{F}_p\bg X_\alpha\jg$ for $\beta\in\Delta\setminus\{\alpha\}$. So we compute 
\begin{align*}
\mathbb{F}_p\bg X_\alpha\jg\otimes_{\mathbb{F}_p\bg X_\alpha\jg,\varphi_\alpha}D_\alpha=\mathbb{F}_p\bg X_\alpha\jg\otimes_{\mathbb{F}_p\bg X_\alpha\jg,\varphi_\alpha}\bigcap_{\beta\in\Delta\setminus\{\alpha\}}\left(E^{sep}_{\Delta\setminus\{\alpha\}}\bg X_\alpha\jg\otimes_{E_{\Delta}}D\right)^{\varphi_\beta=\id}=\\
=\bigcap_{\beta\in\Delta\setminus\{\alpha\}}\left(\mathbb{F}_p\bg X_\alpha\jg\otimes_{\mathbb{F}_p\bg X_\alpha\jg,\varphi_\alpha}E^{sep}_{\Delta\setminus\{\alpha\}}\bg X_\alpha\jg\otimes_{E_{\Delta}}D\right)^{\varphi_\beta=\id}=\\
=\bigcap_{\beta\in\Delta\setminus\{\alpha\}}\left(E^{sep}_{\Delta\setminus\{\alpha\}}\bg X_\alpha\jg\otimes_{E_{\Delta}}\mathbb{F}_p\bg X_\alpha\jg\otimes_{\mathbb{F}_p\bg X_\alpha\jg,\varphi_\alpha}D\right)^{\varphi_\beta=\id}\cong\\
\cong\bigcap_{\beta\in\Delta\setminus\{\alpha\}}\left(E^{sep}_{\Delta\setminus\{\alpha\}}\bg X_\alpha\jg\otimes_{E_{\Delta}}D\right)^{\varphi_\beta=\id}=D_\alpha\ .
\end{align*}
\end{proof}
\emph{Step 5. We show the essential surjectivity of $\mathbb{D}$ here.} Now we apply $\mathbb{V}_{F,\alpha}=(E_\alpha^{sep}\otimes_{\mathbb{F}_p\bg X_\alpha\jg}\cdot)^{\varphi_\alpha=\id}$ to $D_\alpha$ to obtain a finite dimensional $\mathbb{F}_p$-representation $V$ of $\GQpD$. Moreover, we have $\dim_{\mathbb{F}_p}V=\dim_{\mathbb{F}_p\bg X_\alpha\jg}D_\alpha=\rk_{E_\Delta}D$ by the isomorphism \eqref{inftyisom} since $\mathbb{V}_{F,\alpha}$ is rank-preserving by Fontaine's classical result. Using again the isomorphism \eqref{inftyisom} we conclude that the upper horizontal map in the diagram
\begin{align*}
\xymatrix{
E^{sep+}_{\Delta\setminus\{\alpha\}}\bs X_\alpha\js[X_\Delta^{-1}]\otimes_{\mathbb{F}_p\bg X_\alpha\jg}D_\alpha\ar@{^{(}->}[d]\ar[r] & E^{sep+}_{\Delta\setminus\{\alpha\}}\bs X_\alpha\js[X_\Delta^{-1}]\otimes_{E_{\Delta}}D\ar@{^{(}->}[d]\\
E^{sep}_{\Delta\setminus\{\alpha\}}\bg X_\alpha\jg\otimes_{\mathbb{F}_p\bg X_\alpha\jg}D_\alpha\ar[r]^\sim & E^{sep}_{\Delta\setminus\{\alpha\}}\bg X_\alpha\jg\otimes_{E_{\Delta}}D
}
\end{align*}
induced by the containment $D_\alpha\subset E^{sep+}_{\Delta\setminus\{\alpha\}}\bs X_\alpha\js[X_\Delta^{-1}]\otimes_{E_{\Delta}}D$ is injective since so are the vertical arrows as $E^{sep+}_{\Delta\setminus\{\alpha\}}\bs X_\alpha\js[X_\Delta^{-1}]$ is a subring in $E^{sep}_{\Delta\setminus\{\alpha\}}\bg X_\alpha\jg$ and $D$ (resp.\ $D_\alpha$) is flat over $E_\Delta$ (resp.\ over $\Fp\bg X_\alpha\jg$) by Prop.\ \ref{proj} (resp.\ since $\Fp\bg X_\alpha\jg$ is a field). Applying $E^{sep}_\alpha\otimes_{\mathbb{F}_p\bg X_\alpha\jg}\cdot$ we deduce another injective composite map
\begin{align*}
E^{sep}_\Delta\otimes_{\mathbb{F}_p}V\hookrightarrow\\
\hookrightarrow\left(E^{sep+}_{\Delta\setminus\{\alpha\}}\bs X_\alpha\js[X_\Delta^{-1}]\otimes_{\mathbb{F}_p\bg X_\alpha\jg}E^{sep}_\alpha\right)\otimes_{\mathbb{F}_p}V\cong\\
\cong E^{sep+}_{\Delta\setminus\{\alpha\}}\bs X_\alpha\js[X_\Delta^{-1}]\otimes_{\mathbb{F}_p\bg X_\alpha\jg}E^{sep}_\alpha\otimes_{\mathbb{F}_p\bg X_\alpha\jg}D_\alpha=\\
=E^{sep}_\alpha\otimes_{\mathbb{F}_p\bg X_\alpha\jg}E^{sep+}_{\Delta\setminus\{\alpha\}}\bs X_\alpha\js[X_\Delta^{-1}]\otimes_{\mathbb{F}_p\bg X_\alpha\jg}D_\alpha\hookrightarrow\\
\hookrightarrow \left(E^{sep}_\alpha\otimes_{\mathbb{F}_p\bg X_\alpha\jg}E^{sep+}_{\Delta\setminus\{\alpha\}}\bs X_\alpha\js[X_\Delta^{-1}]\right)\otimes_{E_{\Delta}}D\ .
\end{align*}
Taking $\HQpD$-invariants of this inclusion we deduce an inclusion $\mathbb{D}(V)\hookrightarrow D$ using Lemma \ref{complinv}. However, this is an isomorphism by Prop.\ 2.1 in \cite{MultVar} as $\mathbb{D}(V)$ and $D$ have the same rank.
\end{proof}
\begin{rems}
\begin{enumerate}
\item Even though we have constructed $V$ in the proof of the above theorem by a different procedure from just putting $V:=\mathbb{V}(D)$, we still have an isomorphism $V\cong \mathbb{V}(\mathbb{D}(V))\cong \mathbb{V}(D)$ by Prop.\ \ref{mathbbD}.
\item If $\kappa$ is a finite extension of $\Fp$, then we have an equivalence of categories between $\mathrm{Rep}_\kappa(\GQpD)$ and $\mathcal{D}^{et}(\varphi_{\Delta},\Gamma_\Delta,\kappa\otimes_{\Fp}E_\Delta)$. Indeed, we have a natural isomorphism $(\kappa\otimes_{\Fp}E^{sep}_\Delta)\otimes_\kappa\cdot\cong E^{sep}_\Delta\otimes_{\Fp}\cdot$ as functors on $\mathrm{Rep}_\kappa(\GQpD)$.
\end{enumerate}
\end{rems}

\section{The case of $p$-adic representations}

\subsection{Cohomological preliminaries}

We will need the following multivariable analogue of Hilbert's Theorem 90 (additive form).

\begin{pro}\label{hilbert90D}
The continuous group cohomology $H^1_{cont}(\HQpD,E^{sep}_\Delta)$ vanishes.
\end{pro}
\begin{proof}
By Prop.\ \ref{invEsepD} it suffices to show that for finite Galois extensions $E'_\alpha/E_\alpha$ (for all $\alpha\in\Delta$) with Galois group $H'_\alpha:=\Gal(E'_\alpha/E_\alpha)$ we have $H^1(H',E'_\Delta)=\{1\}$ where we put $H':=\prod_{\alpha\in\Delta} H'_\alpha$. Choose a normal basis $e_1,\dots,e_{n_\alpha}\in E'_\alpha$ over $E_\alpha$ for each $\alpha\in\Delta$. By Lemma \ref{E'Dfingen} the set $\{\prod_{\alpha\in\Delta}e_{i_\alpha}\mid 1\leq i_\alpha\leq n_\alpha,\ \alpha\in\Delta\}$ is a basis of the free $E_\Delta$-module $E'_\Delta$. In particular, $E'_\Delta\cong E_\Delta[H']$ is induced as an $H'$-module whence the cohomology group $H^1(H',E'_\Delta)$ is trivial.
\end{proof}

Let $D$ be an abelian group admitting an action of the commutative monoid $\prod_{\alpha\in\Delta}\varphi_\alpha^{\mathbb{N}}$. Fix a total ordering $<$ on $\Delta$ and consider the complex
\begin{align*}
\Phi^\bullet(D)\colon 0\to D\to \bigoplus_{\alpha\in\Delta}D\to \dots\to \bigoplus_{\{\alpha_1,\dots,\alpha_r\}\in \binom{\Delta}{r}}D\to\dots \to D\to 0
\end{align*}
where for all $0\leq r\leq |\Delta|-1$ the map $d_{\alpha_1,\dots,\alpha_r}^{\beta_1,\dots,\beta_{r+1}}\colon D\to D$ from the component in the $r$th term corresponding to $\{\alpha_1,\dots,\alpha_r\}\subseteq \Delta$ to the component corresponding to the $(r+1)$-tuple $\{\beta_1,\dots,\beta_{r+1}\}\subseteq\Delta$ is given by
\begin{equation*}
d_{\alpha_1,\dots,\alpha_r}^{\beta_1,\dots,\beta_{r+1}}=\begin{cases}0&\text{if }\{\alpha_1,\dots,\alpha_r\}\not\subseteq\{\beta_1,\dots,\beta_{r+1}\}\\ (-1)^{\varepsilon}(\id-\varphi_\beta)&\text{if }\{\beta_1,\dots,\beta_{r+1}\}=\{\alpha_1,\dots,\alpha_r\}\cup\{\beta\}\ ,\end{cases}
\end{equation*}
where $\varepsilon=\varepsilon(\alpha_1,\dots,\alpha_r,\beta)$ is the number of elements in the set $\{\alpha_1,\dots,\alpha_r\}$ smaller than $\beta$. Since the operators $(\id-\varphi_\beta)$ commute with each other, $\Phi^\bullet(D)$ is a chain complex of abelian groups. Note that for each $\alpha\in\Delta$ we have a complex
\begin{equation*}
\Phi^\bullet_\alpha(D)\colon 0\to D\overset{\id-\varphi_\alpha}{\to} D\to 0
\end{equation*}
such that $\Phi^\bullet(E^{sep}_\Delta)$ is a kind of completed tensor product of the complexes $\Phi^\bullet_\alpha(E^{sep}_\alpha)$. More precisely, the tensor product over $\mathbb{F}_p$ of the complexes $\Phi^\bullet(E^{sep}_\alpha)$ is the complex $\Phi^\bullet(E^{sep}_{\Delta,\circ})$ which is therefore acyclic in nonzero degrees with $0$th cohomology equal to $\mathbb{F}_p$ by the K\"unneth formula. Note that there are no higher $\Tor$'s as the tensor product is taken over the field $\mathbb{F}_p$. We need the following completed version of this observation.
\begin{pro}\label{phiacyclicEsepD}
The complex $\Phi^\bullet(E^{sep}_\Delta)$ is acyclic in nonzero degrees with $0$th cohomology equal to $\mathbb{F}_p$.
\end{pro}

The following Lemma is well-known.

\begin{lem}\label{inverseconvergephi}
For any finite separable extension $E'_\alpha/E_\alpha$ the map $\id-\varphi_\alpha\colon X'_\alpha E'^+_\alpha\to X'_\alpha E'^+_\alpha$ is bijective.
\end{lem}
\begin{proof}
The kernel of $\id-\varphi_\alpha$ is $\mathbb{F}_p$ which is not contained in $X'_\alpha E'^+_\alpha$. On the other hand, $\sum_{n=0}^\infty\varphi_\alpha^n$ converges on this set and is therefore an inverse to $\id-\varphi_\alpha$ for formal reasons.
\end{proof}

Our key is the following

\begin{lem}\label{phialphasurjEsepD}
For all $\alpha\in S\subseteq\Delta$ the map $\id-\varphi_\alpha\colon E^{sep}_S\to E^{sep}_S$ is surjective with kernel $E^{sep}_{S\setminus\{\alpha\}}$.
\end{lem}
\begin{proof}
We may assume $S=\Delta$. The inclusion $E^{sep}_{\Delta\setminus\{\alpha\}}\subseteq \Ker(\id-\varphi_\alpha)$ is clear. For a collection $E_\beta\leq E'_\beta=\mathbb{F}_{q_\beta}\bg X'_\beta\jg$ ($\beta\in\Delta$) of finite separable extensions the ring $E'_\Delta$ is embedded into $(E'_{\Delta\setminus\{\alpha\}}\otimes_{\mathbb{F}_p}\mathbb{F}_{q_\alpha})\bg X'_\alpha\jg$. By comparing the coefficients we find that $(E'_{\Delta\setminus\{\alpha\}}\otimes_{\mathbb{F}_p}\mathbb{F}_{q_\alpha})\bg X'_\alpha\jg^{\varphi_\alpha=\id}=E'_{\Delta\setminus\{\alpha\}}$.

For the surjectivity pick an element $c$ in $E'_\Delta\subset E^{sep}_\Delta$ for some collection of finite separable extensions $E_\beta\leq E'_\beta=\mathbb{F}_{q_\beta}\bg X'_\beta\jg$ ($\beta\in\Delta$). There exists an integer $k\geq 0$ such that $c$ lies in $X_\Delta^{-k}E'^+_\Delta=\widehat{\bigotimes}_{\beta\in\Delta,\mathbb{F}_p}X_\beta^{-k}E'^+_\beta$. So we may write $c$ as a convergent sum $c=\sum_{n=1}^\infty c_{\overline{\alpha},n}\otimes c_{\alpha,n}$ such that $c_{\overline{\alpha},n}\in X_{\Delta\setminus\{\alpha\}}^{-k}E'^+_{\Delta\setminus\{\alpha\}}$ with $c_{\overline{\alpha},n}\to 0$ and $c_{\alpha,n}\in X_\alpha^{-k}E'^+_\alpha$. The set $X_\alpha^{-k}E'^+_\alpha/X'_\alpha E'^+_\alpha$ is finite, so we choose a finite set $U\subset X_\alpha^{-k}E'^+_\alpha$ of representatives of all the cosets in $X_\alpha^{-k}E'^+_\alpha/X'_\alpha E'^+_\alpha$. We adjoin the roots of the polynomial $f_u(X)=X^p-X-u$ to $E'_\alpha$ for each $u\in U$ in order to obtain a finite separable extension $E''_\alpha/E'_\alpha$ (noting that these polynomials do not have multiple roots). We deduce that each $u\in U$ lies in the image of $\id-\varphi_\alpha\colon E''_\alpha\to E''_\alpha$, and by construction we may write $c_{\alpha,n}=u_n+v_n$ with $u_n\in U$ and $v_n\in X'_\alpha E'^+_\alpha$ for all $n\geq 1$. By Lemma \ref{inverseconvergephi}, the elements $v_n$ are in the image of $\id-\varphi_\alpha$, too, whence so are the elements $c_{\alpha,n}$ by the additivity of the map $\id-\varphi_\alpha$, ie.\ $c_{\alpha,n}=d_{\alpha,n}-\varphi_\alpha(d_{\alpha,n})$ for some  $d_{\alpha,n}\in E''_\alpha$ for all $n\geq 1$. Moreover, the $X_\alpha$-adic valuation of $d_{\alpha,n}$ is bounded by that of the $X_\alpha$-adic valuation of $c_{\alpha,n}$ showing that the sum $d:=\sum_{n=1}^\infty c_{\overline{\alpha},n}\otimes d_{\alpha,n}$ defines an element in $E^{sep}_\Delta$ with $c=d-\varphi_\alpha(d)$.
\end{proof}

\begin{proof}[Proof of Prop.\ \ref{phiacyclicEsepD}]
We proceed by induction on $|\Delta|$. The case $|\Delta|=1$ is clear, so suppose $n:=|\Delta|>1$ and we have proven the statement for any proper subset $S\subsetneq \Delta=\{\alpha_1,\dots,\alpha_n\}$. Let $c=(c_S)_{S\in\binom{\Delta}{r}}\in \bigoplus_{S\in \binom{\Delta}{r}} E^{sep}_\Delta$ be a cocycle in degree $r$. By Lemma \ref{phialphasurjEsepD} we find an element $x=(x_U)_{U\in \binom{\Delta}{r-1}}$ with $x_U=0$ for all $U$ with $\alpha_n\in U$ such that $(c-d^{r-1}(x))_S=0$ for all $S\in \binom{\Delta}{r}$ with $\alpha_n\in S$. Indeed, the map $\cdot\cup\{\alpha_n\}\colon \binom{\Delta\setminus\{\alpha_n\}}{r-1}\to \{S\in\binom{\Delta}{r}\mid \alpha_n\in S\}$ is a bijection and by our assumption that $x$ is concentrated into $\binom{\Delta\setminus\{\alpha_n\}}{r-1}\subset \binom{\Delta}{r-1}$ only the $S\setminus\{\alpha\}$-component of $x$ contributes to the $S$ component of $d^{r-1}(x)$ for $\alpha_n\in S$. So by replacing $c$ with $c-d^{r-1}(x)$ we may assume without loss of generality that $c_S=0$ for all $S$ containing $\alpha_n$. In particular, for $S'\in\binom{\Delta\setminus\{\alpha_n\}}{r}$ we compute 
\begin{align*}
0=(d^r(c))_{S'\cup\{\alpha_n\}}=(-1)^r(\id-\varphi_{\alpha_n})(c_{S'})+\sum_{\beta\in S'}(-1)^{\varepsilon(\beta,S)}(\id-\varphi_\beta)(c_{S'\cup\{\alpha_n\}\setminus\{\beta\}})=\\
=(-1)^r(\id-\varphi_{\alpha_n})(c_{S'})\ .
\end{align*}
Using Lemma \ref{phialphasurjEsepD} again this yields $c_{S'}\in E^{sep}_{\Delta\setminus\{\alpha_n\}}$ for all $S'\in \binom{\Delta}{r}$. Now the statement follows by induction.
\end{proof}

The association $D\mapsto \Phi^\bullet(D)$ is an exact functor from the category of abelian groups with an action of $\prod_{\alpha\in\Delta}\varphi_\alpha^{\mathbb{N}}$ to the category of chain complexes of abelian groups. In particular, for any short exact sequence $0\to D_1\to D_2\to D_3\to 0$, we have a short exact sequence $0\to \Phi^\bullet(D_1)\to \Phi^\bullet(D_2)\to \Phi^\bullet(D_3)\to 0$ of chain complexes. This yields a long exact sequence
\begin{equation*}
0\to h^0\Phi^\bullet(D_1)\to h^0\Phi^\bullet(D_2)\to h^0\Phi^\bullet(D_3)\to h^1\Phi^\bullet(D_1)\to h^1\Phi^\bullet(D_2)\to h^1\Phi^\bullet(D_3)\to\cdots
\end{equation*}
of abelian groups.

\subsection{The multivariable $p$-adic coefficient ring}

Our goal in this section is to lift $E_\Delta$ and $E^{sep}_\Delta$ to characteristic $0$ so we can classify $p$-adic representations of $\GQpD$. Recall \cite{FO} that $\mathcal{O_E}\cong \varprojlim_h \mathbb{Z}/(p^h)\bg X\jg$ is constructed as a Cohen ring of $E\cong\mathbb{F}_p\bg X\jg$. Via the embedding $X\mapsto [\varepsilon]-1$ these are subrings of $\tilde{B}$ which is defined as $\tilde{B}:=W(\widehat{E^{sep}})[p^{-1}]$ where $W(\widehat{E^{sep}})$ is the ring of $p$-typical Witt vectors of the completion $\widehat{E^{sep}}$ (with respect to the $X$-adic topology) of the separable closure $E^{sep}$. Here $[\varepsilon]$ denotes the Teichm\"uller representative of the sequence $\varepsilon=(\varepsilon_n)_n\in \varprojlim_{x\mapsto x^p}\mathcal{O}_{\mathbb{C}_p}\cong \widehat{E^{sep}}^+$ of $p$-power roots of unity with $\varepsilon_1\neq 1$. Note that $\widehat{E^{sep}}$ is an algebraically closed field of characteristic $p$ which is, in fact, isomorphic to the tilt $\mathbb{C}_p^\flat=\Frac(\varprojlim_{x\mapsto x^p}\mathcal{O}_{\mathbb{C}_p}/(p))$ of $\mathbb{C}_p$ in the modern terminology. Further, for any finite extension $E'/E$ contained in $E^{sep}$ there exists a unique finite unramified extension $\mathcal{E}'$ of $\mathcal{E}=\mathcal{O_E}[p^{-1}]$ contained in $\tilde{B}$ with residue field $E'$ (Prop.\ 4.20 in \cite{FO}).

We define the ring $\OED$ as the projective limit $\varprojlim_h \left(\mathbb{Z}/(p^h)\bs X_\alpha\mid \alpha\in\Delta\js [X_\Delta^{-1}]\right)$ and put $\mathcal{E}_\Delta:=\OED[p^{-1}]$ so we have $\OED/(p)\cong E_\Delta$. The Iwasawa algebra $\OED^+=\Zp\bs X_\alpha\mid \alpha\in\Delta\js\leq \OED$ is isomorphic to the completed tensor product of the one-variable Iwasawa algebras $\mathcal{O}_{\mathcal{E}_\alpha}^+:=\Zp\bs X_\alpha\js$ ($\alpha\in\Delta$) over $\Zp$. This motivates the way we can lift $E'_\Delta$ to characteristic $0$ for a collection $E'_\alpha/E_\alpha$ ($\alpha\in\Delta$) of finite separable extensions. We define $$\mathcal{O}_{\mathcal{E}'_\Delta}^+:=\widehat{\bigotimes_{\alpha\in\Delta,\Zp}}\mathcal{O}_{\mathcal{E}'_\alpha}^+$$
as a completed tensor product. If we write $E'_\alpha=\mathbb{F}_{q_\alpha}\bg X'_\alpha\jg$ ($\alpha\in\Delta$) then we may identify $\mathcal{O}_{\mathcal{E}'_\Delta}^+$ with the power series ring $\left(\bigotimes_{\alpha\in\Delta,\Zp}W(\mathbb{F}_{q_\alpha})\right)\bs X'_\alpha\mid \alpha\in\Delta\js$ over the finite \'etale $\Zp$-algebra $\bigotimes_{\alpha\in\Delta,\Zp}W(\mathbb{F}_{q_\alpha})$. We define $\mathcal{O}_{\mathcal{E}'_\Delta}$ as the $p$-adic completion $\widehat{\mathcal{O}_{\mathcal{E}'_\Delta}^+[X_\Delta^{-1}]}=\varprojlim_h \mathcal{O}_{\mathcal{E}'_\Delta}^+[X_\Delta^{-1}]/(p^h)$ and put $\mathcal{E}'_\Delta:=\mathcal{O}_{\mathcal{E}'_\Delta}[p^{-1}]$. We have the following alternative characterization of $\mathcal{O}_{\mathcal{E}'_\Delta}$.

\begin{lem}
Writing $\Delta=\{\alpha_1,\dots,\alpha_n\}$ we have $$\mathcal{O}_{\mathcal{E}'_\Delta}\cong \mathcal{O}_{\mathcal{E}'_{\alpha_1}}\otimes_{\mathcal{O}_{\mathcal{E}_{\alpha_1}}}(\cdots (\mathcal{O}_{\mathcal{E}'_{\alpha_n}}\otimes_{\mathcal{O}_{\mathcal{E}_{\alpha_n}}}\OED))\ .$$
In particular, $\mathcal{O}_{\mathcal{E}'_\Delta}$ is a free module of rank $\prod_{i=1}^n|E'_{\alpha_i}:E_{\alpha_i}|$ over $\OED$.
\end{lem}
\begin{proof}
Each $\mathcal{O}_{\mathcal{E}'_{\alpha_i}}$ is naturally a subring in $\mathcal{O}_{\mathcal{E}'_\Delta}$ and so is $\OED$. Therefore there is a ring homomorphism from the right hand side to the left hand side which is an isomorphism modulo $p$ by Lemma \ref{E'Dfingen}. The first statement follows from the $p$-adic completeness of both sides.

Since $\mathcal{O}_{\mathcal{E}_{\alpha_i}}$ is a complete discrete valuation ring, $\mathcal{O}_{\mathcal{E}'_{\alpha_i}}$ is finite free over $\mathcal{O}_{\mathcal{E}_{\alpha_i}}$ of rank $|E'_{\alpha_i}:E_{\alpha_i}|$ ($i=1,\dots,n$). Therefore the second statement.
\end{proof}

Now we define $\mathcal{E}^{ur}_\Delta:=\varinjlim \mathcal{E}'_\Delta$ and $\mathcal{O}_{\mathcal{E}^{ur}_\Delta}:=\varinjlim \mathcal{O}_{\mathcal{E}'_\Delta}$ where $E'_\alpha$ runs over the finite subextensions of $E_\alpha$ in $E^{sep}_\alpha$ for all $\alpha\in\Delta$. Further, we denote by $\widehat{\mathcal{E}^{ur}_\Delta}$ (resp.\ by $\mathcal{O}_{\widehat{\mathcal{E}^{ur}_\Delta}}$) the $p$-adic completion of $\mathcal{E}^{ur}_\Delta$ (resp.\ of $\mathcal{O}_{\mathcal{E}^{ur}_\Delta}$). We have $\mathcal{O}_{\widehat{\mathcal{E}^{ur}_\Delta}}/(p)\cong E^{sep}_\Delta$ by construction. The group $\GQpD$ acts naturally on $\widehat{\mathcal{E}^{ur}_\Delta}$ (resp.\ on $\mathcal{O}_{\widehat{\mathcal{E}^{ur}_\Delta}}$). Moreover, for each $\alpha\in\Delta$ we have the Frobenius lift $\varphi_\alpha$ on $\tilde{B}_\alpha$ (the copy of $\tilde{B}$ indexed by $\alpha$) which acts on $[\varepsilon]$ by raising to the $p$th power (as it is a Teichm\"uller representative). So we have $\varphi_\alpha(X_\alpha)=(X_\alpha+1)^p-1$. For each finite extension $E'_\alpha/E_\alpha$ we have $\varphi_\alpha(\mathcal{E}'_\alpha)\subset \mathcal{E}'_\alpha$, so this defines an action of $\varphi_\alpha$ on the rings $\mathcal{E}^{ur}_\Delta$, $\mathcal{O}_{\mathcal{E}^{ur}_\Delta}$, $\widehat{\mathcal{E}^{ur}_\Delta}$, and $\mathcal{O}_{\widehat{\mathcal{E}^{ur}_\Delta}}$ for all $\alpha\in\Delta$. These operators commute with each other and with the action of the group $\GQpD$.

\begin{pro}\label{padicinv}
We have 
\begin{eqnarray*}
\widehat{\mathcal{E}^{ur}_\Delta}^{\HQpD}=\mathcal{E}_\Delta\ ,\qquad &&\bigcap_{\alpha\in\Delta}\widehat{\mathcal{E}^{ur}_\Delta}^{\varphi_\alpha=\id}=\mathbb{Q}_p\ ,\text{ and}\\
\mathcal{O}_{\widehat{\mathcal{E}^{ur}_\Delta}}^{\HQpD}=\OED\ , \qquad &&\bigcap_{\alpha\in\Delta}\mathcal{O}_{\widehat{\mathcal{E}^{ur}_\Delta}}^{\varphi_\alpha=\id}=\Zp\ .
\end{eqnarray*}
\end{pro}
\begin{proof}
The statements on $\widehat{\mathcal{E}^{ur}_\Delta}$ follow from those on $\mathcal{O}_{\widehat{\mathcal{E}^{ur}_\Delta}}$ as $p$ is $\varphi_\alpha$- and $\HQpD$-invariant for all $\alpha\in\Delta$. Moreover, the latter statements are consequences of Prop.\ \ref{invEsepD}, resp.\ Lemma \ref{phiinv} using devissage.
\end{proof}

\subsection{The equivalence of categories}

We denote by $\RepDZp$ (resp.\ by $\RepDQp$) the category of continuous representations of $\GQpD$ on finitely generated $\Zp$-modules (resp.\ on finite dimensional $\mathbb{Q}_p$-vector spaces). Let $T$ (resp.\ $V$) be an object in $\RepDZp$ (resp.\ in $\RepDQp$). We define
\begin{equation*}
\mathbb{D}(T):=\left(\mathcal{O}_{\widehat{\mathcal{E}^{ur}_\Delta}}\otimes_{\Zp}T\right)^{\HQpD}\qquad\left(\text{resp.\ } \mathbb{D}(V):=\left(\widehat{\mathcal{E}^{ur}_\Delta}\otimes_{\Qp}V\right)^{\HQpD} \right)\ .
\end{equation*}

By Prop.\ \ref{padicinv} $\mathbb{D}(T)$ (resp.\ $\mathbb{D}(V)$) is a module over $\OED$ (resp.\ over $\mathcal{E}_\Delta$). Moreover, it admits an action of the monoid $T_{+,\Delta}$: the action of $\varphi_\alpha$ ($\alpha\in\Delta$) is trivial on $T$ (resp.\ on $V$) and therefore comes from the action on $\mathcal{O}_{\widehat{\mathcal{E}^{ur}_\Delta}}$ (resp.\ on $\widehat{\mathcal{E}^{ur}_\Delta}$) defined above. The action of $\Gamma_\Delta=\GQpD/\HQpD$ comes from the diagonal action of $\GQpD$ on $\mathcal{O}_{\widehat{\mathcal{E}^{ur}_\Delta}}\otimes_{\Zp}T$ (resp.\ on $\widehat{\mathcal{E}^{ur}_\Delta}\otimes_{\Qp}V$).

\begin{pro}\label{mathbbDZp}
Let $T$ be an object in $\RepDZp$. The natural map
\begin{equation*}
\mathcal{O}_{\widehat{\mathcal{E}^{ur}_\Delta}}\otimes_{\OED} \mathbb{D}(T)\to \mathcal{O}_{\widehat{\mathcal{E}^{ur}_\Delta}}\otimes_{\Zp}T
\end{equation*}
is an isomorphism.
\end{pro}
\begin{proof}
This is very similar to the proof of Prop.\ 2.30 in \cite{FO}. We proceed in two steps. Assume first that $T$ is killed by a power $p^h$ of $p$. We use induction on $h$. The case $h=1$ is done in Prop.\ \ref{mathbbD}. Now for $h>1$ we have a short exact sequence $0\to T_1\to T\to T_2\to 0$ of objects in $\RepDZp$ such that $pT_1=0$ and $p^{h-1}T_2$. Since $\mathcal{O}_{\widehat{\mathcal{E}^{ur}_\Delta}}$ has no $p$-torsion, it is flat as $\Zp$-module. Therefore we obtain a short exact sequence
\begin{equation*}
0\to \mathcal{O}_{\widehat{\mathcal{E}^{ur}_\Delta}}\otimes_{\Zp}T_1\to \mathcal{O}_{\widehat{\mathcal{E}^{ur}_\Delta}}\otimes_{\Zp}T\to \mathcal{O}_{\widehat{\mathcal{E}^{ur}_\Delta}}\otimes_{\Zp}T_2\to 0\ .
\end{equation*}
Now we have an identification $\mathcal{O}_{\widehat{\mathcal{E}^{ur}_\Delta}}\otimes_{\Zp}T_1\cong E^{sep}_\Delta\otimes_{\Fp}T_1\cong E^{sep}_\Delta\otimes_{E_\Delta}\mathbb{D}(T_1)$. In particular, as a representation of $\HQpD$ we have $\mathcal{O}_{\widehat{\mathcal{E}^{ur}_\Delta}}\otimes_{\Zp}T_1\cong (E^{sep}_\Delta)^{\dim_{\Fp} T_1}$. In particular, Prop.\ \ref{hilbert90D} yields $H^1_{cont}(\HQpD,\mathcal{O}_{\widehat{\mathcal{E}^{ur}_\Delta}}\otimes_{\Zp}T_1)=\{1\}$. By the long exact sequence of continuous $\HQpD$-cohomology we deduce the exactness of the sequence
\begin{equation*}
0\to \mathbb{D}(T_1)\to \mathbb{D}(T)\to \mathbb{D}(T_2)\to 0\ .
\end{equation*}
Now we have a commutative diagram
\begin{equation*}
\xymatrix{
0\ar[r] & \mathcal{O}_{\widehat{\mathcal{E}^{ur}_\Delta}}\otimes_{\OED}\mathbb{D}(T_1)\ar[r]\ar[d]^\sim & \mathcal{O}_{\widehat{\mathcal{E}^{ur}_\Delta}}\otimes_{\OED}\mathbb{D}(T)\ar[r]\ar[d] & \mathcal{O}_{\widehat{\mathcal{E}^{ur}_\Delta}}\otimes_{\OED}\mathbb{D}(T_2)\ar[r]\ar[d]^\sim & 0\\
0\ar[r] & \mathcal{O}_{\widehat{\mathcal{E}^{ur}_\Delta}}\otimes_{\Zp}T_1\ar[r] & \mathcal{O}_{\widehat{\mathcal{E}^{ur}_\Delta}}\otimes_{\Zp}T\ar[r] & \mathcal{O}_{\widehat{\mathcal{E}^{ur}_\Delta}}\otimes_{\Zp}T_2\ar[r] & 0
}
\end{equation*}
with exact rows. Thus the vertical map in the middle is an isomorphism by induction using the $5$-lemma.

The general case follows from this by taking the projective limit of the isomorphisms above for $T/p^hT$ as $h$ tends to infinity.
\end{proof}

An \'etale $T_{+,\Delta}$-module over $\OED$ is a finitely generated $\OED$-module $D$ together with a semilinear action of the monoid $T_{+,\Delta}$ such that for all $\varphi_t\in T_{+,\Delta}$ the map 
\begin{equation*}
\id\otimes\varphi_t\colon \varphi_t^*D:=\OED\otimes_{\OED,\varphi_t}D\to D
\end{equation*}
is an isomorphism. We denote by $\mathcal{D}^{et}(\varphi_{\Delta},\Gamma_\Delta,\OED)$ the category of \'etale $T_{+,\Delta}$-modules over $\OED$. As in the mod $p$ case, $\mathcal{D}^{et}(\varphi_{\Delta},\Gamma_\Delta,\OED)$ has the structure of a neutral Tannakian category. If $D$ is a finitely generated $\OED$ module that is killed by a power $p^h$ of $p$ we define the generic length of $D$ as $\genlength D:=\sum_{i=1}^h\rk_{E_\Delta}p^{i-1}D/p^iD$ where $\rk_{E_\Delta}$ denotes the generic rank (ie.\ dimension over $\Frac(E_\Delta)$ of the localisation at $(0)$).

\begin{cor}
The functor $\mathbb{D}$ is exact. $\mathbb{D}(T)$ is an object in $\mathcal{D}^{et}(\varphi_{\Delta},\Gamma_\Delta,\OED)$ for any $T$ in $\RepDZp$. Moreover, if $T$ is killed by a power of $p$ then the we have $\genlength\mathbb{D}(T)=\length_{\Zp}T$.
\end{cor}
\begin{proof}
If $T$ is an object in $\RepDZp$ such that $p^h T=0$, then we have $H^1(\HQpD, \mathcal{O}_{\widehat{\mathcal{E}^{ur}_\Delta}}\otimes_{\Zp}T)=\{1\}$ by induction on $h$ using the long exact sequence of continuous $\HQpD$-cohomology. So the exactness of $\mathbb{D}$ on finite length objects in $\RepDZp$ follows the same way as in the proof of Prop.\ \ref{mathbbDZp} in the special case when $pT_1=0$. Now if $0\to T_1\to T_2\to T_3\to 0$ is an arbitrary short exact sequence in $\RepDZp$ then we have an exact sequence 
\begin{align*}
0\to T_1[p^h]\to T_2[p^h]\to T_3[p^h]\overset{\partial_h}{\to} T_1/p^hT_1\to T_2/p^hT_2\to T_3/p^hT_3\to 0
\end{align*}
of finite length objects for all $h\geq 1$. Applying $\mathbb{D}$ yields an exact sequence
\begin{align*}
0\to \mathbb{D}(T_1[p^h])\to \mathbb{D}(T_2[p^h])\to \mathbb{D}(T_3[p^h])\to \mathbb{D}(T_1/p^hT_1)\to \mathbb{D}(T_2/p^hT_2)\to \mathbb{D}(T_3/p^hT_3)\to 0
\end{align*}
for all $h\geq 1$. Since $T_i$ is finitely generated over $\Zp$, we have $T_i[p^h]=(T_i)_{tors}$ for $h\geq h_0$ large enough ($i=1,2,3$). In particular, the connecting map $T_i[p^{(n+1)h}]\overset{p^h\cdot}{\to}T_i[p^{nh}]$ is the zero map for $h\geq h_0$ and $i=1,2,3$. Thus the Mittag--Leffler property is satisfied for both $\mathrm{Im}(\partial_h)_h$ and $\Coker(\partial_h)_h$ as the map $T_1/p^{h+1}T_1\to T_1/p^hT_1$ is surjective for all $h\geq 1$. Hence taking the projective limit we obtain an exact sequence $0\to \mathbb{D}(T_1)\to \mathbb{D}(T_2)\to \mathbb{D}(T_3)\to 0$ as claimed.

The statement on the generic length follows from the exactness using Prop.\ \ref{mathbbD} and induction on $h$ such that $p^hT=0$. In particular, $\mathbb{D}(T)$ is finitely generated over $\OED$ if $T$ has finite length. Now if $T$ is not necessarily of finite length then we apply the exactness of $\mathbb{D}$ on the exact sequence $0\to T[p]\to T\overset{p\cdot}{\to} T\to T/pT\to 0$ to obtain that $\mathbb{D}(T/pT)=\mathbb{D}(T)/p\mathbb{D}(T)$ which is finitely generated over $E_\Delta$. Therefore $\mathbb{D}(T)$ is finitely generated over $\OED$ by the $p$-adic completeness of $\mathbb{D}(T)$ (it follows easily from the definition that we have $\varprojlim_h\mathbb{D}(T/p^hT)=\mathbb{D}(T)$).

Finally, the \'etale property for finite length modules follows by induction on the length from the case $h=1$ (Prop.\ \ref{mathbbD}) and in general by taking the projective limit.
\end{proof}

Conversely, let $D$ be an object in $\mathcal{D}^{et}(\varphi_{\Delta},\Gamma_\Delta,\OED)$. We define
\begin{equation*}
\mathbb{T}(D):=\bigcap_{\alpha\in\Delta}\left(\mathcal{O}_{\widehat{\mathcal{E}^{ur}_\Delta}}\otimes_{\OED}D\right)^{\varphi_\alpha=\id}\ .
\end{equation*}
This is a $\Zp$-module admitting a diagonal action of $\GQpD$ via the formula $g(\lambda\otimes d):=g(\lambda)\otimes\chi(g)(d)$ where $\chi\colon \GQpD\twoheadrightarrow\Gamma_\Delta$ is the quotient map.

\begin{pro}
For any object $D$ in $\mathcal{D}^{et}(\varphi_{\Delta},\Gamma_\Delta,\OED)$, the natural map
\begin{equation*}
\mathcal{O}_{\widehat{\mathcal{E}^{ur}_\Delta}}\otimes_{\Zp}\mathbb{T}(D)\to \mathcal{O}_{\widehat{\mathcal{E}^{ur}_\Delta}}\otimes_{\OED}D
\end{equation*}
is an isomorphism.
\end{pro}
\begin{proof}
This is completely analogous to the proof of Prop.\ 2.31 in \cite{FO}. We proceed in two steps. At first assume that $p^hD=0$ for some integer $h\geq 1$. Consider the exact sequence $0\to D[p]\to D\to D/D[p]\to 0$ and apply the exact functor $\Phi^\bullet\circ(\mathcal{O}_{\widehat{\mathcal{E}^{ur}_\Delta}}\otimes_{\OED}\cdot)$ to obtain an exact sequence
\begin{equation*}
0\to \Phi^\bullet(\mathcal{O}_{\widehat{\mathcal{E}^{ur}_\Delta}}\otimes_{\OED}D[p])\to \Phi^\bullet(\mathcal{O}_{\widehat{\mathcal{E}^{ur}_\Delta}}\otimes_{\OED}D)\to \Phi^\bullet(\mathcal{O}_{\widehat{\mathcal{E}^{ur}_\Delta}}\otimes_{\OED}D/D[p])\to 0\ .
\end{equation*}
By Thm.\ \ref{modpequiv} $D[p]$ is in the image of the functor $\mathbb{D}$ whence $\mathcal{O}_{\widehat{\mathcal{E}^{ur}_\Delta}}\otimes_{\OED}D[p]$ is isomorphic to $(E^{sep}_\Delta)^{\rk_{E_\Delta}D[p]}$ as a $\prod_{\alpha\in\Delta}\varphi_\alpha^{\mathbb{N}}$-module using Prop.\ \ref{mathbbD}. In particular, $h^1\Phi^\bullet(\mathcal{O}_{\widehat{\mathcal{E}^{ur}_\Delta}}\otimes_{\OED}D[p])=0$ by Prop.\ \ref{phiacyclicEsepD}. This yields an exact sequence
\begin{equation*}
0\to \mathbb{T}(D[p])\to \mathbb{T}(D)\to \mathbb{T}(D/D[p])\to 0\ ,
\end{equation*}
and the statement follows the same way as in the proof of Prop.\ \ref{mathbbDZp}.

The general case follows by taking the limit.
\end{proof}

Now note that $\mathbb{T}(D)$ is finitely generated over $\Zp$: this is obvious in the case when $p^hD=0$ using induction on $h$ and in the general case by Nakayama's lemma as we have $\mathbb{T}(D)=\varprojlim_h \mathbb{T}(D/p^hD)$ by construction. So we deduce
\begin{thm}\label{equivcatZp}
The functors $\mathbb{D}$ and $\mathbb{T}$ are quasi-inverse equivalences of categories between the Tannakian categories $\RepDZp$ and $\mathcal{D}^{et}(\varphi_{\Delta},\Gamma_\Delta,\OED)$.
\end{thm}

Finally, an \'etale $T_{+,\Delta}$-module over $\mathcal{E}_\Delta$ is a finitely generated $\mathcal{E}_\Delta$-module $D$ together with a semilinear action of the monoid $T_{+,\Delta}$ such that there exists an object $D_0$ in $\mathcal{D}^{et}(\varphi_{\Delta},\Gamma_\Delta,\OED)$ with an isomorphism $D\cong D_0[p^{-1}]= \mathcal{E}_\Delta\otimes_{\OED}D_0$. We denote by $\mathcal{D}^{et}(\varphi_{\Delta},\Gamma_\Delta,\mathcal{E}_\Delta)$ the category of \'etale $T_{+,\Delta}$-modules over $\mathcal{E}_\Delta$. As before, $\mathcal{D}^{et}(\varphi_{\Delta},\Gamma_\Delta,\mathcal{E}_\Delta)$ has the structure of a neutral Tannakian category. We have the following characteristic $0$ version of the category equivalence:

\begin{thm}\label{equivcatQp}
The functors
\begin{eqnarray*}
V&\mapsto&\mathbb{D}(V):=\left(\widehat{\mathcal{E}^{ur}_\Delta}\otimes_{\Qp}V\right)^{\HQpD}\\
D&\mapsto&\mathbb{V}(D):=\bigcap_{\alpha\in\Delta}\left(\widehat{\mathcal{E}^{ur}_\Delta}\otimes_{\mathcal{E}_\Delta}D\right)^{\varphi_\alpha=\id}
\end{eqnarray*}
are quasi-inverse equivalences of categories between the Tannakian categories $\RepDQp$ and $\mathcal{D}^{et}(\varphi_{\Delta},\Gamma_\Delta,\mathcal{E}_\Delta)$.
\end{thm}
\begin{proof}
Since $\GQpD$ is compact, any finite dimensional $\Qp$-representation $V$ contains a $\GQpD$-invariant lattice $T$. The statement follows from Thm.\ \ref{equivcatZp} by inverting $p$ on both sides. The compatibility with tensor products and duals follows the same way as in characteristic $p$.
\end{proof}

\begin{rems}
\begin{enumerate}
\item If $A$ is a $\Zp$-algebra which is finitely generated as a module over $\Zp$, then we have an equivalence of categories between $\mathrm{Rep}_A(\GQpD)$ and $\mathcal{D}^{et}(\varphi_{\Delta},\Gamma_\Delta,A\otimes_{\Zp}\OED)$. Indeed, we have a natural isomorphism $(A\otimes_{\Zp}\mathcal{O}_{\widehat{\mathcal{E}^{ur}_\Delta}})\otimes_A\cdot\cong \mathcal{O}_{\widehat{\mathcal{E}^{ur}_\Delta}}\otimes_{\Zp}\cdot$ as functors on $\mathrm{Rep}_A(\GQpD)$. Similarly, if $K$ is a finite extension of $\Qp$, then we have an equivalence of categories between $\mathrm{Rep}_K(\GQpD)$ and $\mathcal{D}^{et}(\varphi_{\Delta},\Gamma_\Delta,K\otimes_{\Qp}\mathcal{E}_\Delta)$.
\item It is expected that there is a similar equivalence of categories for representations of the $|\Delta|$th direct power of the group $\Gal(\overline{\mathbb{Q}_p}/F)$ for a finite extension $F/\mathbb{Q}_p$. However, at this point it is not clear what type of $(\varphi,\Gamma)$-modules one should consider. The usual cyclotomic $(\varphi,\Gamma)$-modules do not seem to be well-suited for the purpose of the $p$-adic and mod $p$ Langlands programme. On the other hand, the Lubin--Tate setting may not work properly in characteristic $p$ due to the non-existence of the distinguished left inverse $\psi$ of $\varphi$. To work over the character variety of the group $\mathcal{O}_F$ \cite{BSX} seems, however, to be a good candidate.
\end{enumerate}
\end{rems}

\end{document}